\documentclass{article}
\usepackage{natbib}

\usepackage{amssymb}
\usepackage{amsmath}
\usepackage{amsfonts}
\usepackage{amsthm}

\title{A Reduced Form for Linear Differential Systems and its Application to Integrability of Hamiltonian Systems }

\author{Ainhoa Aparicio-Monforte\footnote{RISC, Johannes Kepler University. Altenberger Strasse 69 A-4040 Linz, Austria.} \thanks{First author supported both by a FEDER doctoral grant of  R\'egion Limousin and by the Austrian FWF grant Y464-N18.} 
\& Jacques-Arthur Weil\footnote{XLIM (CNRS \& Universit\'e de Limoges) - 123, avenue Albert Thomas - 87060 Limoges Cedex}
}

\newtheorem{thm}{Theorem}
\newtheorem{rem}[]{Remark}
\newtheorem{defn}[thm]{Definition}
\newtheorem{prop}[thm]{Proposition}
\newtheorem{cor}[thm]{Corollary}
\newtheorem{exmp}[thm]{Example}
\newtheorem{lem}[thm]{Lemma}

\begin{document}
\maketitle
\begin{abstract}
Let $k$ be a differential field with algebraic closure $\bar{k}$, and let  $[A]:\; Y'=AY$ with  $A\in \mathcal{M}_n (k)$ be a  linear differential system. Denote by $\mathfrak{g}$  the Lie algebra of the differential Galois group of $[A]$.
We say that a matrix $R\in \mathcal{ M}_{n}(\overline{k})$ is  a \emph{reduced form} of $[A]$ if $R\in \mathfrak{g}(\overline{k})$ and
there exists $P\in GL_n (\overline{k})$ such that $R=P^{-1}(AP-P')\in  \mathfrak{g}(\overline{k})$.
Such a form is often the sparsest possible attainable through gauge transformations without introducing new transcendents.
In this article, we discuss how to compute reduced forms of some symplectic differential systems, arising as variational equations of hamiltonian systems. We use this to give an effective form of the Morales-Ramis theorem on (non-)-integrability of Hamiltonian systems.
\end{abstract}

\section{Introduction}

This article lies at the crossroads of
differential Galois theory and the complete (Liouville) integrability of Hamiltonian systems.
Let  $[A]:\; Y'=AY$ with  $A\in \mathcal{M}_n (k)$ be a  linear differential system,
where $k$ denotes a differential field of characteristic zero with algebraically closed constant field $C$.

On the differential Galois theory side, we propose (following works of Kolchin, Kovacic, Singer, Mitschi and others)  a notion of reduced form for the system $[A]$. Let $G$ denote the differential Galois group of $Y'=AY$ and $\mathfrak{g}$ its Lie algebra.
We say that a matrix $R\in {\cal M}_{n}(\overline{k})$ is  a {\em reduced form} of $[A]$ if $R\in \mathfrak{g}(\overline{k})$ and
there exists $P\in GL_n (\overline{k})$ such that $R=P^{-1}(AP-P')$. 
Such a form turns out to be very natural and somehow the most concise attainable through gauge transformations without introducing new transcendents.
When a system is in reduced form,  many of its intrinsic properties can be readily computed. In our case, we are mainly  concerned with detecting the (non-)abelianity  of the Lie algebra of the differential Galois group of $[A]$.
Stemming from works by Kovacic and Kolchin (\cite{Ko99a} or  \cite{Ko69a,Ko71a})  on the inverse problem of differential Galois theory, this notion of \textit{reduced form} has been since developed by 
	\cite{MiSi96a,MiSi02a,CoMiSi05a} (for inverse problems
as well) and many others  (e.g \cite{Ha05b,JuLe07a}, see more references in \cite{Si09a,PuSi03a}, generally for inverse problems). We explore here both some of the advantages and the constructibility of this reduced form applied to the context of direct problems in differential Galois theory.

On the Hamiltonian system side, the Morales-Ramis theorem states that if a Hamiltonian system is meromorphically completely (Liouville) integrable, then the Lie algebra of its variational equation (a linear differential system) along a non-constant integral curve is abelian.
Applying the above notion of reduced form (effectively), we are able to decide whether this Lie algebra is abelian or not, and thus propose new (non-)integrability criteria.
When the Lie algebra is abelian, our reduced form is useful to simplify the study of higher variational equations in order to apply  the Morales-Ramis-Sim\'o theorem (this is explored in \cite{AW11a,Ap10a}).

Many of the existing examples are Hamiltonian systems with two degrees of freedom. In such case, we will show (or recall to specialists) how the normal variational equation can be put into a reduced form using Kovacic's algorithm. In the literature, many authors take into consideration only the normal variational equation $[N]$. Since the Lie algebra $\mathfrak{g}_N$ of $[N]$  is only a quotient of $\mathfrak{g}$,  the data thus obtained cannot be complete, especially as far as the abelianity  of $\mathfrak{g}$ is concerned.

In this paper, we give an algorithm which, by looking for a reduced form of a linear differential system, will either show that the full system has non-abelian Lie algebra (hence proving non-integrability of the original differential system) or (in the abelian case) return a reduced form for the variational equation.

Our construction is quite systematic and should generalize to many other systems (such generalizations are initiated in \cite{AW11a}). It also shows the lovely simple structure of reduced forms. This will be very useful in studying higher variational equations.
Systems occurring in higher variational equations are reducible and of big sizes; computing their Galois groups is difficult. Although a general decision procedure exists (\cite{CoSi99a} in completely reducible cases, \cite{Hr02a} in general), studies towards extensive descriptions of differential Galois groups in reducible cases (e.g \cite{Be01a, Ha05a}) are still in progress.
The case of two completely reducible factors (\cite{BeSi99a,Be02a}) is the only one which seems to be fully understood.
We propose an approach which focuses on finding abelianity criteria for the Lie algebra instead of computing the Galois group.

This paper is split into four sections apart from this introduction. Section 2 deals with theoretical background material such as differential Galois theory and integrability of Hamiltonian systems. Section 3 introduces the concept of reduced form as well as its importance and usefulness, showing that the Kovacic algorithm can be re-used to put second order systems into a reduced form. In Section 4, we consider linear differential systems whose matrix  $A$ lies in $\mathfrak{sp}(4,k)$ (typically the variational equation of a Hamiltonian system with two degrees of freedom) and provide a reduction algorithm together with an abelianity criterion.
In section 5, we show how to apply our techniques to questions of integrability of Hamiltonian systems. In particular, our method allows us to reprove that the lunar Hill system  is not meromorphically integrable.\\
\noindent{\bf Acknowledgements}: we warmly thank S. Simon for fruitful conversations and for 
suggesting the study of the Hill example which is given in section 5.
We also thank Thomas Cluzeau, Elie Compoint, Maria Przybylska, Andrzej Maciejewski and particularly the referees for their useful suggestions.
\section{Background material}\label{background}
This section contains necessary background material and no new results.
\subsection{Some differential Galois theory}
General references for this section are \cite{PuSi03a,Si09a} and many others, cited therein.
Let $(\,k\,,\, '\,)$ be a differential field with an algebraically closed field of constants $C$ of characteristic zero.
Let $[A]:\; Y'=AY$ denote a linear  differential system with $A\in M_n (k)$.
A Picard-Vessiot field $K$ is a minimal differential extension of $k$ generated by the entries of a
fundamental solution matrix $U$ of $[A]$. The {\em differential Galois group} of $A$ is  the group of the automorphisms over $K$ that leave $k$ invariant and that commute with the derivation:
$$
G:=\partial{}Aut (K/k)=\left\lbrace\sigma\,:\, K\mapsto K\,:\, \begin{array}{c}\forall u\in k, \sigma(u) = u \text{ and } \\
\forall u\in K, \;  \sigma(u')=\sigma(u)' \end{array}\right\rbrace 
$$
The Galois group $G$ is a linear algebraic group acting on the vector space of solutions of $[A]$ and as such admits a faithful representation in $GL_n(C)$ once a fundamental solution matrix $U$ is chosen. 
The {\em Lie algebra of $G$} is
$\mathfrak{g}:=Lie(G)=T_{e}G$, a $C$-vector space of matrices endowed with a Lie algebra structure by the
usual Lie bracket $[A,B]:=AB-BA$. The Lie algebra $\mathfrak{g}$ is abelian if and only if the connected component of the identity $G^\circ$ is abelian as well.

 We say that two systems $Y'=AY$ and $Z'=BZ$ with $A,\, B\in M_n (k)$ are {\em gauge equivalent} if there is a linear change of variable $Y=PZ$ with $P\in GL_n (k)$ such that $Z' =BZ$, i.e:
   $$B=P[A]:= P^{-1}(AP-P') \quad\hbox{ (gauge transformation)}.
   $$
    Two gauge equivalent systems share the same Galois group $G$. Moreover,   $V:=P^{-1}U$ is a fundamental solution matrix of $[B]$ ; 
    given $\sigma\in G$ with matrix $M_\sigma$, we have $\sigma(U)=UM_\sigma$ and $\sigma(P^{-1})=P^{-1}$ so $\sigma(V)=P^{-1}UM_\sigma=VM_\sigma$, i.e the representation of $G$ is unchanged on this new system.
     If the coefficients of the gauge transformation belong to the algebraic closure
     $\overline{k}$,  then $G$ is altered whereas $\mathfrak{g}$ and $G^\circ$ remain unchanged.

\subsection{Hamiltonian Systems}
Let $(M,\omega)$ be a complex analytic symplectic manifold of complex dimension $2n$.
By the Darboux theorem, we know that, locally, $M$ is isomorphic to an open domain $U\subset\mathbb{C}^{2n}$ and that, taking locally suitable coordinates $(q,p)=(q_1,\ldots,q_n,p_1,\ldots,p_n)$, we can 
deal with associated entities (such as HamiltonÕs equations and the Poisson
bracket) in terms of the matrix
$$
J:=\left[\begin{array}{cc}0_n & I_n\\ -I_n & 0_n\end{array}\right].
$$
In these coordinates, given a function $H\in C^{2}(U)\,:\,U\rightarrow \mathbb{C}$,
we define a Hamiltonian system over $U\subset\mathbb{C}^{2n}$ as the differential system given by the vector field
$X_H = J\cdot\nabla H $:
\begin{equation}\label{Sham}
\dot{q}_i =\frac{\partial H}{\partial p_i}(q,p)\quad\hbox{,}
\quad\dot{p}_i =-\frac{\partial H}{\partial q_i}(q,p)
\quad\hbox{for}\quad i=1\ldots n.
\end{equation}
Let  $z(t)$ be a parametrization of an integral curve
$\Gamma\subset U$ that satisfies (\ref{Sham}).
The Hamiltonian $H$ is constant over those integral curves. Indeed,
$$X_H\cdot H := \langle \nabla H \,,\, X_H\rangle = \langle \nabla H \,,\, J\nabla H\rangle=0.$$
Therefore integral curves will lie on the levels of energy of $H$.
A function $F : U\longrightarrow\mathbb{C}$, meromorphic over $U$, is called a \textit{meromorphic first integral} of
 (\ref{Sham}) if $X_H \cdot F = 0$ (i.e. $F$ is constant over the integral curves of $H$). 
 The Poisson bracket $\left\lbrace\,,\,\right\rbrace$ of two meromorphic functions $f,g$ defined over
 a symplectic manifold 
 is defined as
$
\left\lbrace f,g\right\rbrace :=\langle X_f, \nabla g\rangle =\langle -X_g ,\nabla f \rangle   $
  and, in coordinates,
\begin{equation}
\left\lbrace f,g\right\rbrace = \sum^{n}_{i=1}\frac{\partial f}{\partial q_i}\frac{\partial g}{\partial p_i}-\frac{\partial f}{\partial p_i}\frac{\partial g}{\partial q_i}.
\end{equation}
The Poisson bracket endows  the set of first integrals of (\ref{Sham}) with a Lie algebra structure~;
in fact, a function $F$ is a first integral of (\ref{Sham})  if and only if it is in involution with $H$, i.e. $\left\lbrace F,H\right\rbrace =0$.

A Hamiltonian system is called {\em meromorphically Liouville  integrable} if it possesses $n$  first integrals $H_1 = H,\ldots,H_n$ meromorphic over $U$ satisfying:
\begin{itemize}
\item[-] they are functionally independent: $\nabla H_1,\ldots,\nabla H_n$ are linearly independent over $U$,
\item[-] they are in involution: $\left\lbrace H_i\,,\, H_j\right\rbrace = 0$ for $i,j=1\ldots n$.
\end{itemize}

The theory of Morales and Ramis, developing on founding works of Ziglin and followers (\cite{Zi82a,Zi83a}, \cite{It85a}, 
\cite{Yo86a,Yo87a,Yo87b,Yo88a}
\cite{BaChRoSi96a,ChRoSi95a}, \cite{Mo99a}),
aims at proving rigorously non-integrability using differential Galois groups (or monodromy groups) of variational equations.

\subsubsection{Variational equations}
Let $\Gamma\subset U$ be an integral curve of (\ref{Sham})
	parametrized by $z(t)$.
The differential field $k:=\mathbb{C}\langle z(t)\rangle$ will be called the \emph{coefficient field} (or, informally, the \emph{field of rational functions}).
We define the \textit{variational equation} of (\ref{Sham}) along $\Gamma$ as the linearization of (\ref{Sham}) along $z(t)$. It describes the behavior of the solutions of (\ref{Sham}) near $z(t)$.
In other words, if $z_{0}(t)$ and $z_1(t):=z_{0}(t)+Y(t)$ where $Y(t)$ is an infinitesimal perturbation of $z_0(t)$, then
a first order approximation gives
	$Y'=d_{z_0 (t)}X_H Y(t)$.
Taking coordinates, the variational equation along $z(t)$ is  given by
\begin{equation}\label{VE}
Y'=AY \quad \hbox{\rm with }\quad A:=J\cdot \mathrm{Hess}(H)(z(t)).
\end{equation}
As the Hessian $\mathrm{Hess}(H)(z(t))$ is a symmetric matrix,  we have
$A\in\mathfrak{sp}(2n,k)$ (i.e $A^{T}\cdot J + J \cdot A = 0$). Thus, as recalled in the Appendix,  there exists a fundamental matrix of solutions 
$U$ of  $Y' = AY$  in a Picard-Vessiot field $K$ such that   $U\in \mathrm{Sp}(2n, K)$. 
In the sequel, we will work with such a symplectic  fundamental solution matrix $U(t)\in \mathrm{Sp}(2n , K)$.

We denote by $G$ the differential Galois group of the variational equation (\ref{VE}) and by $\mathfrak{g}$ its Lie algebra.
As system (\ref{VE}) is Hamiltonian ($ A \in\mathfrak{sp}(2n, k)$),  $G$ is a subgroup of  $\mathrm{Sp}(2n,\mathbb{C})$
and $\mathfrak{g}\subset\mathfrak{sp}(2n,\mathbb{C})$ (see next section). 

\begin{thm}[Morales and Ramis, see \cite{Mo99a} Theorem 4.2 p.81]\label{MR}
Let $z(t)$ be a non-singular integral curve of  the Hamiltonian system (\ref{Sham}). Let  (\ref{VE})  be its variational equation along $z(t)$.
If (\ref{Sham}) is  meromorphically Liouville integrable then the Lie algebra $\mathfrak{g}$ of (\ref{VE})  is abelian.
\end{thm}
This theorem is in fact a non-integrability criterion: when checking whether $\mathfrak{g}$ is abelian only negative answers are conclusive.
Indeed, there are non-integrable systems for which the Lie algebra (of the differential Galois group) of their variational equation along a given integral curve is abelian.
In this case, one is bound to consider higher order variational equations, see e.g \cite{MoRaSi07a} section 3.4 or \cite{AW11a,Ap10a}.

\subsubsection{The Normal variational equation}\label{222}

In general, studying the abelianity of $\mathfrak{g}$  is simpler than finding either a complete set of meromorphic first integrals of (\ref{Sham}) or an obstruction to their existence without taking into account (\ref{VE}). Still, checking straightforwardly the abelianity of $\mathfrak{g}$  is not an easy matter.

 However, we can take advantage of the fact that $z'(t)$ is a particular solution of the variational equation (\ref{VE}) along $z(t)$. Indeed,  Proposition 4.2 p 76 in \cite{Mo99a} (see also our Appendix) ensures the existence of a symplectic  gauge transformation  that allows us to reduce this variational equation (i.e. to rule out one degree of freedom) and to obtain the {\em normal variational equation} $\mathrm{(NVE)}$.  
 In the new coordinates, $\mathrm{(NVE)}$ can be written as $Z'=NZ$ where $N\in\mathfrak{sp}(2(n-1) ; k)$:   (NVE) is therefore yet another (Hamiltonian) linear differential system.

 Consider now $U_N\in \mathrm{Sp}(2(n-1), K_N)$  a fundamental matrix of solutions of $Z'=NZ$,  where $K_N\supset k$ is a Picard Vessiot extension for (NVE). We have $k\subset K_N \subset K$ so $\tilde{G}:=\partial Aut_{K_N}(K)\vartriangleleft G:=\partial Aut_{k}(K)$.
Since the differential Galois group of the (NVE) is given by $G_N:=\partial Aut_{k}(K_N) \simeq
G/\tilde{G}$
then the Lie Algebra of $G_N$ is  $\mathfrak{g}_N = \mathfrak{g}/\tilde{\mathfrak{g}}$. Thus, if  $\mathfrak{g}_N$ is non-abelian then $\mathfrak{g}$ is not abelian either. In fact, the usual method, ever since Morales introduced it, consists in reducing the (\ref{VE}) to its (NVE) and then argue about the abelianity of $\mathfrak{g}_N$. If $\mathfrak{g}_N$ turns out to be abelian, no conclusion for (\ref{VE}) is reached and so the higher order variational equations are explored. We will improve this procedure in sections 3 and 4.
\\
Let us illustrate the notion of  normal variational equation by taking a general example for $2n=4$. In such a case, the variational equation (\ref{VE}) can be written:
\begin{equation}\label{VEn=2}
Y'= {A}Y\quad \hbox{with}\quad A:=
{\small\left[\begin{array}{cccc}\alpha_{11}&\alpha_{12}&\alpha_{13}&\alpha_{14}\\ \alpha_{21}&\alpha_{22}&\alpha_{14}&\alpha_{24}\\ \alpha_{31}&\alpha_{32}&-\alpha_{11}&-\alpha_{21}\\\alpha_{32}&\alpha_{42}&-\alpha_{12}&
-a_{22}\end{array}\right] }\in\mathfrak{sp}(4, k).
\end{equation}
	Take a particular solution $z=(z_1,z_2,z_3,z_4)^T$ of (\ref{Sham}) (for $n=2$).
Then $z'=(z'_1 ,z'_2,z'_3,z'_4)^T$ is a particular solution of (\ref{VEn=2}). Completing $z'$ to a symplectic basis  of $k^4$ (see Appendix) yields
the symplectic change of variable $Z= {P}U $, where
\begin{equation}\label{usualgauge}
{\small
 {P}:=\left[\begin{array}{cccc}z'_1 & 0 & 0 & 0 \\
							 z'_2 & 1 & 0 & 0 \\
							 z'_3 & \frac{z'_4}{z'_1}& \frac{1}{z'_1}&-\frac{z'_2}{z'_1}\\
							 z'_4 & 0 & 0 & 1	\end{array}\right]}\in Sp(4,k).
\end{equation}
It induces a
(symplectic) gauge transformation $A_N:= {P}[ {A}]= {P}^{-1}({A} {P}- {P}')\in\mathfrak{sp}(4;k)$, where

\begin{equation}\label{A_N}
A_N:=\left[\begin{array}{cccc}0 &a_{12,n} & a_{13,n} & a_{14,n} \\ 0 & n_{11} & a_{14,n} &n_{12}\\ 0 & 0 & 0 & 0\\ 0 & n_{21}& -a_{12,n}& -n_{11} \end{array}\right].
\end{equation}

The \emph{algebraic normal variational} (or for convenience \emph{normal variational equation}, see Remark \ref{NVE}) \emph{equation} is given by $\dot{U} = N U$ where $$N:=\left[\begin{array}{cc}n_{11} & n_{12} \\ n_{21} & -n_{11}\end{array}\right].$$

\section{Reduced form of a linear differential system}
\subsection{A Kovacic reduced form}
The study of the integrability of dynamical systems has originated numerous deep studies on local normal forms near equilibrium points (e.g \cite{Bi66a} and references therein) or even along periodic solutions (e. g. \cite{Ko96a}). However, we do not know of a similar global notion properly defined along non-constant  regular solutions.

In this work, we propose to explore a weaker global notion, a notion of {\em  reduced form } of a linear differential system which is strongly inspired by the work of Kovacic and Kolchin 
	(\cite{Ko69a,Ko71a}, \cite{Ko99a}) and more recent works like \cite{MiSi96a,MiSi02a} on the inverse problem in differential Galois theory.

\begin{defn}
 Let $A$ belong to $ M_{n}(\bar{k})$. Let $G$ be the differential Galois group of $\dot{Y} = AY$ and $\mathfrak{g}$ its Lie algebra. We say that $A$ is in {\em reduced form} if $A\in\mathfrak{g}(\bar{k})$.
\end{defn}

The expression  $A\in\mathfrak{g}(k)$  is read ``$A$ is a $k$-point of $\mathfrak{g}$''. The Lie algebra $\mathfrak{g}$ is a vector space of dimension $d$
spanned by a set of matrices $M_1 ,\ldots, M_d \subset {\mathcal{M}_n }(\mathbb{C})$.    The set $\mathfrak{g}(k)$ of {\em $k$-points of $\mathfrak{g}$}  is  defined as
 $$
\mathfrak{g}(k):=\{f_1 M_1 + \cdots + f_d M_d,\quad f_i \in k\}.
$$

\noindent Similarly, given a linear algebraic group $G$, the set $G(k)$ of $k$-points of $G$ is the set of matrices whose entries are in $k$ and satisfy the defining equations of $G$. The question is whether a reduced form exists. If $k$ is a $C_1$-field\footnote{A field $k$ is called quasi-algebraically closed (or $\mathrm{C}_1$) if every non-constant homogeneous polynomial $P\in k[X_{1},\ldots,X_{n}]$  of degree less than $n$ has a non-trivial zero in $k^{n}$ (\cite{La52a}). An instance of one such field is $k=\mathbb{C}(x)$. In addition, by Theorem 5 of \cite{La52a} we know that the algebraic closure of a $C_1$ field is $C_1$ as well.}  then the following result due to Kolchin and Kovacic shows that the answer is positive:
\begin{thm}[see \cite{Ko71a} or \cite{PuSi03a} p.25 Corollary 1.32]\label{Kovacic}
Let $k$  be a differential $C_1$-field. Let $A\in\mathcal{M}_n (k)$ and assume that the differential Galois Group $G$ of the system $Y'=AY$ is connected. Let $\mathfrak{g}$ be the Lie algebra of $G$. Let $H$ be a connected algebraic group such that
its Lie algebra $\mathfrak{h}$  satisfies $A\in\mathfrak{h}(k)$. 
Then $G\subset H$ and there exists $P\in H(k)$ such that the equivalent  differential equation 
$F'= R F$, with $Y=PF$ and $R =P^{-1}(AP- P')$, satisfies $R \in \mathfrak{g}(k)$.
\end{thm}

\begin{cor}\label{kovacic-corollary}
We use the same notations as above but now assume that $G$ is {\em not} connected.
Let $K$ be a Picard-Vessiot extension of $k$ for $Y'=AY$.
Let $k_1$ denote the algebraic closure of $k$ in $K$.
Then there exists
	$P\in H(k_{1})$
such that the equivalent differential equation $f'=Rf$ with $Y=Pf$ and $R:=P^{-1}(AP- P')$ satisfies
	$R \in \mathfrak{g}(k_1)$.
\end{cor}
\begin{proof}
By virtue of \cite{PuSi03a} (Proposition 1.34 p. 26) the connected component of $G$ containing the identity $G^{\circ}$, satisfies $G^\circ = \partial Aut_{k_1}(K)$.  As  $k_1$ is an algebraic extension of $k$, it is still a $C_1$-field. We now pick $k_1$ as a base field; then $K$ is a Picard Vessiot extension of $k_1$ with Galois group $G^\circ$ satisfying the hypotheses of Theorem \ref{Kovacic}.
\end{proof}
Thus, for a $C_1$-field $k$  and $A\in \mathcal{M}_n(k)$,  a reduced form of $[A]$ will be given by a gauge equivalent matrix $R\in\mathfrak{g}(k_1)$, i.e $R:=P[A]$ for some $P\in GL_n (k_1)$. The change of variable given by $P$ is algebraic~;  the related gauge transformation does not introduce transcendental coefficients, hence preserving $G^\circ$ and $\mathfrak{g}$.
\begin{defn}
Consider a differential system $Y'=AY$. We keep the notations of Theorem \ref{Kovacic} and corollary
\ref{kovacic-corollary}. A matrix $P\in GL_{n}(k_{1})$ is called a {\em reduction matrix} for $[A]$
if $P[A]$ is in reduced form, i.e $P[A]  \in \mathfrak{g}(k_{1})$.
\end{defn}

In this work, we will show how to make the Kolchin-Kovacic reduction theorem effective for $2\times2$ systems with Galois group in $\mathrm{SL}(2,C)$ and for some systems with group in $\mathrm{Sp}(4,C)$
(obtained via the linearization of Hamiltonian systems).
In both cases, we will see that one does not need the ``$C_{1}$-field'' hypothesis.
There are other types of systems (e.g $3\times3$ systems with Galois group $\mathrm{SO}(3)$) for which effective reduction may be performed but where the ``$C_{1}$-field'' hypothesis cannot be eluded (see chapter 3 of \cite{Ap10a}, also \cite{JuLe07a}).

\subsection{The Lie algebra associated to $A$}
Consider a matrix $A=\left(a_{i,j}\right) \in {\mathcal M}_{n}(k)$. Let $a_1,\ldots,a_r$ denote a basis of the $C$-vector space spanned by the $a_{i,j}$ for $i,j=1,\ldots,n$.
We thus have a decomposition $A:=\sum_{i=1}^{r} a_{i}(t) M_{i}$ where the  $M_{i}$ are constant matrices. Of course, the choice of the $a_{i}$ in this decomposition is not unique, but the $C$-vector space generated by the $M_{i}$ is.\\
 The Lie algebra generated by $M_{1},\ldots,M_{r}$ (i.e generated as a $C$-vector space by the $M_{i}$ and all their iterated Lie brackets) will be called the {\em Lie algebra associated to $A(t)$} and denoted by $Lie(A)$. Its dimension $d$ (as a vector space) satisfies $r\leq d\leq n^{2}$.
With this terminology, a system $Y'=AY$ is in reduced form if $Lie(A)$ is the Lie algebra $Lie(Y'=AY)$ of the differential Galois group.\\
This notion of Lie algebra associated to $A$ appears in works of Magnus and of Wei and Norman, \cite{WeNo63a,WeNo64a}.\\

Pick a gauge transformation $P\in\mathrm{GL}_n (k)$ and let $P[A]:=P^{-1}(AP-P')$;
we say that $P[A]$ is a {\em partial reduction} (and that $P$ is a partial reduction matrix) if $Lie(P[A])\subsetneq Lie(A)$.
As we will see in section 4, a reduced form is obtained by means of a sequence of partial reductions.\\

Consider now a system $Y'=RY$ in reduced form, i.e  $R\in \mathfrak{g}(k)$ where $\mathfrak{g}:=\mathrm{Lie}(Y'=RY)$.
Call $r$ a minimal number of generators of $\mathfrak{g}$ as a Lie algebra, (which is generally less than its dimension $d$ as a vector space)
and let $m$ be the dimension of the $C$-vector space generated by the coefficients of $R$.
We say $R$ is \emph{maximally reduced} if $m=r$.
Note that if $\mathfrak{g}$ is abelian, then a reduced form is automatically maximally reduced (because $r=d$). 

\begin{exmp}
{\em Consider the linear differential system $Y'=AY$ with $$A:={\tiny\left[
\begin{array}{cccc}
0  & f_1  &  f_{3} & f_2\\
0  & 0  &  f_2 &0\\
 0   & 0  & 0  & 0\\
0  & 0  & -f_1  & 0
\end{array}
\right]}
\in \mathcal{M}_4 (k)$$ (where functions $f_{i}\in k$ are linearly independent over $C$)
and assume that this is a reduced form, i.e the Lie algebra of its Galois group is the Lie algebra generated by
$$ M_1:={\tiny \left[
\begin{array}{cccc}
0  & 1  &  0 &0\\
0  & 0  &  0 &0\\
 0   & 0  & 0  & 0\\
0  & 0  & -1  & 0
\end{array}
\right]} \,,\, M_2:={\tiny \left[
\begin{array}{cccc}
0  & 0  &  0 & 1\\
0  & 0  &  1 &0\\
 0   & 0  & 0  & 0\\
0  & 0  & 0  & 0
\end{array}
\right]}\,\text{ and }\,M_3 :={\tiny \left[
\begin{array}{cccc}
0  & 0  &  1 &0\\
0  & 0  &  0 &0\\
 0   & 0  & 0  & 0\\
0  & 0  & 0 & 0
\end{array}
\right]}. $$
We see that $[M_1 , M_2]=2 M_{3}$, which is linearly independent from $M_1$ and $M_2$.
Thus, if we call $\mathfrak{g}$ the Lie algebra generated by $M_{1}$ and $M_{2}$, it is spanned over $C$ by $M_1 , M_2$ and $M_3$
and we have that $\mathrm{dim}_C(\mathfrak{g})=3 \geq 2$.
We see that system $[A]$ is maximally reduced if (and only if) $f_{3}=0$.}$\square$

\end{exmp}

\subsection{Solving abelian reduced systems}\label{Solving abelian reduced systems}

Consider a linear differential system $Y'=M(t)Y$ such that $M(t)\in\mathfrak{h}(k)$ where $\mathfrak{h}$ is the Lie algebra associated to $M(t)$. If $\mathfrak{h}:=\mathrm{span}_{C}\lbrace M_1 ,\ldots , M_d\rbrace$ then we can write $$M=\sum^{d}_{i=1} f_i M_i\quad \text{with}\quad f_i \in k \quad \text{for} \quad i=1\ldots d .$$
As noted by Wei and Norman (who use this  in \cite{WeNo63a,WeNo64a}), if $\mathfrak{h}$ is abelian, solving the system is straightforward.
For instance, suppose that $d=2$ and  that $\mathfrak{h}$ is abelian, i.e. $[M_1\,,\, M_2]=0$. In this situation the system will be of the form  $Y'=(f_1 M_1 + f_2 M_2) Y$. The two differential systems: $Y'_1 = f_{1}M_1 Y_1$ and $Y'_2 = f_{2}M_2 Y_2$ admit respectively $U_1 := \mathrm{exp}(\int f_1 M_1)$ and $U_2 := \mathrm{exp}(\int f_2 M_2)$ as fundamental solution matrices. As $[U_1 \,,\, M_2] = 0$,
a calculation shows that $(U_1 U_2)' = (f_1 M_1 + f_2 M_2) U_1 U_2$. This argument can be extended to $d\geq 2$ by  induction. Therefore, if $\mathfrak{h}$ is abelian  and we consider a system $M$ in reduced form, we only need to solve each separate system $Y'_j = f_j M_j Y_j$: $U=\prod^{d}_{i=1}\mathrm{exp}(\int f_i  M_i)$ is a fundamental solution matrix for the complete system. Such solving methods apply, more generally, when $A$ commutes with $\int A$ or when $Lie(A)$ is solvable (this is the Wei-Norman method exposed in \cite{WeNo63a,WeNo64a}).

Of course, solving is not the ultimate aim of system reduction, but these formulae will be useful in proving the correctness of our reduction procedure.

\subsection{Reduced form when $G\subset SL(2,\mathbb{C})$}\label{RedG2}

The Kovacic algorithm \cite{Ko86a} is an algorithm devised to compute Liouvillian solutions of second order linear
differential equations. In fact, it can be used to compute reduced forms for $Y'=NY$ with $N\in \mathcal{M}_2(k)$,
when $G^\circ _N\subset SL(2,C)$. We recall standard notations for abelian (non-trivial) connected subgroups
of $SL(2,C)$:  the additive group is $ \mathbb{G}_a := {\tiny\left\lbrace\left[\begin{array}{cc} 1 & c \\ 0 & 1\end{array}\right]\,:\,c\in\mathbb{C}  \right\rbrace }$ with Lie algebra
$ \mathfrak{g}_a := {\tiny\left\lbrace\left[\begin{array}{cc} 0 & \alpha \\ 0 & 0\end{array}\right]\,:\,\alpha\in\mathbb{C}  \right\rbrace} $,
and the multiplicative group is $\mathbb{G}_m := {\tiny\left\lbrace\left[\begin{array}{cc} c & 0 \\ 0 & c^{-1}\end{array}\right]\,:\,c\in\mathbb{C}^\star  \right\rbrace}$ with Lie algebra
$\mathfrak{g}_m := {\tiny\left\lbrace\left[\begin{array}{cc} \alpha & 0 \\ 0 & -\alpha\end{array}\right]\,:\,\alpha\in\mathbb{C} \right\rbrace}$.

In what follows, we assume  that the differential Galois group of $[N]$ is in $SL(2,C)$ and $\mathrm{Tr}(N)=0$, i.e
$N\in sl(2,k)$ (otherwise an easy reduction puts us in this form). As a consequence, any fundamental solution
matrix has a constant determinant.

 Let $K_N$ be a Picard-Vessiot  extension of $k$ associated to $Y'=NY$. The next proposition gives a complete reduction procedure, using the Kovacic algorithm, i.e shows how one can compute reduced forms and reduction matrices using the Kovacic algorithm for Liouvillian solutions (and its extension to algebraic solutions in \cite{SiUl93a}); although it will probably not surprise specialists, we include it for completeness and because the reduction matrices will be used in the next section.

\begin{prop}\label{Kovacic algorithm}
Consider a $2\times 2$ linear differential system $Y'=NY$ with  $\mathrm{Tr}(N)=0$ (i.e $N\in sl(2,k)$) and Galois group $G\subset SL(2,C)$.
Then $G\subsetneq SL(2,C)$ if and only if there exists an algebraic extension $k^\circ$ of $k$ such that one of the following (mutually exclusive and to be read in this order) cases holds. The field $k^\circ$ is given by the Kovacic algorithm and its extension by Singer and Ulmer \cite{SiUl93a} (Theorem 4.1) for algebraic solutions.
\begin{enumerate}
\item[Case (1):] The system $[N]$ admits two solutions\footnote{This phrase is an abuse of language which, precisely means: the space of the solutions of $[N]$ that belong to $(k^\circ)^{2}$ is of dimension two.}
$Y_{1},Y_{2}\in (k^\circ)^{2}$. Then $G_{N}$ is finite and $\mathfrak{g}=\{0\}$.
\\
Let $P=(Y_{1},Y_{2}) \in SL(2,k^\circ)$ (after multiplying $Y_{i}$ by a scalar so that $\det(P)=1$).
\\
Then $P[N]=0\in \mathfrak{g}(k^\circ)$ and $P$ is a reduction matrix.

\item[Case (2):] The system $[N]$ admits one solution
$Y_{1}\in (k^\circ)^{2}$. Then $G_{N}^\circ=\mathbb{G}_a$.
\\
Let $P=(Y_{1},F_{2}) \in SL(2,k^\circ)$ (where $F_{2}\in (k^\circ)^2$ is any vector such that $\det(P)=1$).
\\
Then $P[N]=\tiny\left[\begin{array}{cc} 0 & a\\ 0 & 0 \end{array}\right]\in \mathfrak{g}_a(k^\circ)$ and $P$ is a reduction matrix.
\item[Case (3):] There exist $a\in k^\circ$ and $F_{1},F_{2}\in (k^\circ)^{2}$ such that,
if $f$ is a solution of $f'=a f$, $Y_{1}:=f.F_{1}$ and $Y_{2}:=\frac1f F_{2}$ are solutions of $[N]$
(exponential over $k^\circ$).
Then $G_{N}^\circ=\mathbb{G}_m$.
\\
Let $P=(F_{1},F_{2}) \in SL(2,k^\circ)$ (after multiplying $F_{i}$ by a scalar so that $\det(P)=1$).
 Then $P[N]=\tiny\left[\begin{array}{cc} a & 0\\ 0 & -a \end{array}\right] \in \mathfrak{g}_m(k^\circ)$ and $P$ is a reduction matrix.

\item[Case (4):] There exist $a\in k^\circ$ and $F_{1}\in (k^\circ)^{2}$ such that,
if $f$ is a solution of $f'=a f$, $Y_{1}:=f.F_{1}$ is a solution of $[N]$ (exponential over $k^\circ$).
Then $\mathfrak{g}=\left\{ \tiny\left[\begin{array}{cc} c & d\\ 0 & -c \end{array}\right] | c,d\in C\right\}$.
Let $P=(F_{1},F_{2}) \in SL(2,k^\circ)$ (where $F_{2}\in (k^\circ)^2$ is any vector such that $\det(P)=1$).
Then $P$ is a reduction matrix.

\end{enumerate}

\end{prop}

\begin{proof}
We review each case (in the proposed order) and study the Galois group in each case.
Classifications of the algebraic subgroups of $SL(2,C)$ will show that these four cases cover all proper algebraic subgroups of $SL(2,C)$.
\\
Case $(1)$ is the case when the group is finite. Then the proposed $P$ is a fundamental solution matrix. Our hypothesis on $N$ (i.e $Tr(N)=0$) implies that $\det(P)$ is constant. Multiplying one column by a constant, we may secure that $\det(P)=1$. If now $U$ is a fundamental solution of $[N]$ and we write $U=PC$, then
$C'=0$~; as $C$ is a fundamental matrix for $[P[N]]$, this confirms that $P[N]=0$.
\\
In Case $(2)$, $G^\circ$ has an invariant vector $Y_{1}$ so $G^\circ\subset \mathbb{G}_a$ and, since $G$ is not finite, (otherwise case $(1)$), we have $G^\circ = \mathbb{G}_a$. Pick a fundamental matrix $U$ with determinant $1$
and write it as $U=PV$. As $Y_{1}$ is a solution, the first column of $V$ can be chosen to be $(1,0)^T$.
Now, as $\det(V)=1$, we obtain $V=  \tiny\left[\begin{array}{cc} 1 & v\\ 0 & 1 \end{array}\right]$ for some $v\in K$.
As $V'=P[N].V$, the form of $P[N]$ follows.
\\
In Case $(3)$, we have $\mathfrak{g}=\mathfrak{g}_m$. A fundamental solution matrix $U$ can be written
$U=PV$ with $V=\tiny\left[\begin{array}{cc} f & 0\\ 0 & f^{-1} \end{array}\right]$ and the result follows.
\\ In Case $(4)$, the galois group is
$G=\left\{  \tiny\left[\begin{array}{cc} c & d\\ 0 & c^{-1} \end{array}\right]\,|\, c\in C^*,d\in C\right\}$. Calculations are the same as in $(2)$ and $(3)$ and are left to the reader.
\end{proof}

\begin{rem}
\begin{enumerate}
\item In case $(1)$, the field $k^\circ$ can have arbitrarily large degree (in the notations of \cite{SiUl93a}, if
$G=D_{n}^{SL_{2}}$ then $k^\circ$ is of degree $4n$ over $k$).
\\
In case $(2)$,
$G=\left\{  \tiny\left[\begin{array}{cc} \mu^k & d\\ 0 & \mu^{-k} \end{array}\right] \,|\,k\in\{0,..,n-1\},d\in C\right\}$
where $\mu$ is an $n$-th root of unity. Then $Y_{1}=f.F_{1}$ with $F_{1}\in k^2$ and $f^n\in k$.
So $k^\circ:=k(f)$ is a cyclic extension of degree $n$ of $k$.
\\
In case $(3)$, we have $G=\mathbb{G}_m$ (then $k^\circ=k$) or $G=D_{\infty}^{SL_{2}}$ (then
$k^\circ$ is quadratic over $k$). In case $(4)$, $k^\circ=k$.
\item Cases $(1)$, $(2)$ and $(3)$ are the cases when $\mathfrak{g}$ is abelian. In that case, we just saw that
the field $k^\circ$ may have large degree. However, for the applications that we have in mind
(normal variational of Hamiltonian systems), it turns out that $k^\circ=k$ in most examples\footnote{Actually, in {\em all} examples from mechanics that we can think of, but we may have missed a few.} that we know.

\end{enumerate}
\end{rem}

For $2\times 2$ systems, the construction shows that  reduced forms are the most natural and the simplest forms into which a system $N\in M_2(k^\circ)$
(with $\mathfrak{g}_N$ abelian) can be put.
The fields over which such reductions can be performed effectively via the Kovacic algorithm are given in \cite{UlWe96a} (see subsection 4.1).

\section{Reduced form of first variational equations and non-integrability}\label{RedandInt}

Consider a linear differential system $[A]:\; Y'=AY$ with $A\in\mathfrak{sp}(4,k)$.  Call $G$ its Galois group and $\mathfrak{g}$ its Lie algebra.
We further assume that we know a rational solution of $[A]$ (i.e. with coefficients in $k$) so that the material applies to variational equations of Hamiltonian systems along a known non-constant solution.

The aim of this section is to establish an algorithm which either finds an obstruction to the abelianity of $\mathfrak{g}$ or returns a reduced form for $A$ if $\mathfrak{g}$ is abelian. In view of its application to Hamiltonian systems, this would either prove the non-integrability of the system (an effective version of Theorem \ref{MR}) or put the first variational equation in a reduced form so as to simplify the study of higher order variational equations.

\subsection{Admissible base fields}\label{admissiblefields}
In view of practical computation,  we first assume that $k$ is an {\em effective field} (i.e one can perform the arithmetic operations $+,-,*,/$ and algorithmically test when two elements of $k$ are equal).
In order to check effectively the abelianity of $\mathfrak{g}$, we need three algorithmic tools:
\begin{enumerate}
\item The {\em Kovacic algorithm} (for solving second order differential systems, see \cite{Ko86a,UlWe96a}),
\item An algorithm to solve {\em Risch equations}~:\label{Risch}
	for $f,g\in k$, decide whether the equation $y'=f y+g$ has a solution belonging to $k$ (and, if yes, compute it),
\item An algorithm to solve {\em Limited integration} problems~:\label{LIP}
	 for $f,g \in k$, decide whether there is a constant $\beta$ and $h\in k$ such that $f+\beta g=h'$
(and, if yes, compute them).
\end{enumerate}
We now study conditions on $k$ so that all computations required by the above problems can be performed
effectively.

\begin{defn} We say that $k$ is an {\em admissible field} if it is an effective field and if it comes equipped with algorithms to perform the three tasks above: Kovacic algorithm, solving Risch equations and solving limited integration problems.
\end{defn}

Given an operator $L\in k[\partial]$, we say that a solution $y$ of $L(y)=0$ is rational if $y\in k$,
and we say that $y$ is exponential if $y'/y\in k$.
\begin{lem}\label{admissible-field}
Let $(k,\partial)$ be an effective differential field such that, given $L\in k[\partial]$, one can effectively
\begin{enumerate}
\item compute a basis of rational solutions of $L(y)=0$ (i.e solutions in $k$) when such a basis exists, and
\item compute all right-hand first order factors (i.e solutions $y$ such that $y'/y\in k$, exponential solutions) of second order linear differential equations.
\end{enumerate}
Then $k$ is an admissible field.
\end{lem}

\begin{proof}
If $y$ satisfies a limited integration problem, then $y$ is a solution in $k$ of $L(y)=0$ with
$L=LCLM(\partial-f'/f,\partial - g'/g) \partial$ (where $LCLM$ denotes the least common left multiple in
$k[\partial]$, see \cite{PuSi03a} p.50). Indeed, a basis of solutions of $L$ is $1,\int f,\int g$. If now $y$ solves a Risch equation, then a similar calculation shows that  $y$ is a solution in $k$ of $L(y)=0$ with $L=(g\partial - g') (\partial - f)$.
Last, in \cite{UlWe96a}, it is shown that, if the two conditions of the Lemma are fulfilled, then one can apply a full
Kovacic algorithm.
\end{proof}
In \cite{Si91a}, Lemma 3.5, it is shown that if $k$ is an elementary extension
of $C(x)$ or if $k$  is an algebraic extension of a purely transcendental liouvillian extension
of $C(x)$, then $k$ satisfies the conditions of the lemma and hence is an admissible field.

\begin{rem}
In \cite{Si91a}, Theorem 4.1, it is shown that an algebraic extension of an admissible field is also an admissible field.
This is used in the next sections, where an algebraic extension of $k$ is (sometimes) needed in order to perform reduction.
\end{rem}

\subsection{Reduction of the normal variational equation}

Let $K$ denote a Picard-Vessiot extension for $[A]$.
In view of an application to variational equations of Hamiltonian systems,
we assume  that we know a rational solution of $[A]$.
This allows us to apply the gauge transformation given in (\ref{usualgauge}), yielding an equivalent  system of the form
$$\label{S_N}
A_N :=\left[\begin{array}{cccc} 0 & a_{12,n} & a_{13,n} & a_{14,n}
							\\ 0 & n_{11} & a_{14,n} & n_{12}
							\\ 0 & 0 & 0 & 0
							\\ 0 & n_{21} & -a_{12,n} & -n_{11} \end{array}\right] \quad\mathrm{and}\quad
							N:=\left[\begin{array}{cc}
							 n_{11} & n_{12}
						   \\n_{21} & -n_{11}\end{array}\right] \in \mathfrak{sl}(2,k)
$$

\begin{prop}
Using the notations adopted in section \ref{RedG2}, suppose that $\mathfrak{g}_N$ is abelian. Then there exists a symplectic  change of variable  $P_N\in Sp(4,k^\circ)$ that puts $A_N$ into the form $B:=P_N[ A_N ]$   given below:
\begin{eqnarray}
\nonumber\text{Case $\mathfrak{g}_N=\{0\}$ } 
	&:& B=a_{12}M_1 + a_{14}M_2 + a_{13}M_3\\				
\nonumber\text{Case $\mathfrak{g}_N=\mathfrak{g}_a$ }&:& B=a_{12}M_1 + a_{14}M_2 + a_{13}M_3 + a_{24} M_a \text{  with  }  \int a_{24} \notin k^{\circ}\\
\nonumber\text{Case $\mathfrak{g}_N=\mathfrak{g}_m$} &:& B=a_{12}M_1 + a_{14}M_2 + a_{13}M_3 + a_{22}M_m\\
\nonumber&&  \text{  with  } a_{22}=E'/E \text{  where  } E\notin k^{\circ}.
\end{eqnarray}
\noindent where   ${\small M_1 :=\left[\begin{array}{cccc} 0 &  1 & 0 & 0 \\ 0 & 0 & 0 & 0 \\ 0 & 0 & 0 & 0 \\ 0 & 0 & -1 & 0\end{array}\right]}$, 
${\small M_2 :=\left[\begin{array}{cccc} 0 & 0 & 0 & 1 \\ 0 & 0 & 1 & 0 \\ 0 & 0 & 0 & 0 \\ 0 & 0 & 0 & 0\end{array}\right]}$,  
${\small M_3 :=\left[\begin{array}{cccc} 0 & 0 & 1 & 0 \\ 0 & 0 & 0  & 0 \\ 0 & 0 & 0 & 0 \\ 0 & 0 & 0 & 0\end{array}\right]}$, \\
${\small M_a :=\left[\begin{array}{cccc} 0 & 0 & 0 & 0 \\ 0 & 0 & 0 & 1 \\ 0 & 0 & 0 & 0 \\ 0 & 0 & 0 & 0\end{array}\right]}$ 
and ${\small M_m :=\left[\begin{array}{cccc} 0 & 0 & 0 & 0 \\ 0 & 1 & 0 & 0 \\ 0 & 0 & 0 & 0 \\ 0 & 0 & 0 & -1\end{array}\right]}$.
\end{prop}

\begin{proof}
Keeping the notations of the previous section, we write $K_N$ for  the Picard-Vessiot  extension of $Y'=NY$.
Applying the Kovacic algorithm as in proposition \ref{Kovacic algorithm}, we obtain an algebraic extension
$k^\circ$ of $k$ and a reduction matrix $P\in SL(2,k^\circ)$ for $[N]$.
Let $P=(p_{i,j})$ be this reduction matrix.

We extend this to a gauge transformation on $A_{N}$:
$$
P_N := \left[\begin{array}{cccc}1 & 0 & 0 & 0 \\ 0 & p_{11} & 0 & p_{12} \\ 0 & 0 & 1 & 0 \\ 0 & p_{21} & 0 & p_{22}\end{array}\right] \quad\text{where} \left[\begin{array}{cc} p_{11} & p_{12} \\ p_{21} & p_{22}\end{array}\right]\text{ reduces } N.
$$
\noindent The classification in the proposition now follows from the fact that, as $P_{N}$ is (by construction) symplectic, we have $P_{N}[A]\in \mathfrak{sp}(4,k^\circ)$\footnote{
Or by the computation of $A:=P_N [A_N]=P^{-1}_N(A_N P_N - P'_N)$, using the fact that $P_N$ is symplectic and therefore  $P_N^{-1}=-J\; {}^tP_N \; J$ is easily computed}.
\end{proof}
\begin{rem}
These matrices $M_{i}$ satisfy simple relations:
\begin{eqnarray}\label{produitMi}
 M_{1}M_{2}=-M_{2}M_{1}=M_{3} \quad \hbox{\rm and } \; M_{i}M_{j}=0 \; \quad\hbox{\rm otherwise}.
\end{eqnarray}
In particular, $M_{i}^2=0$ so $\exp(M_{i})=Id + M_{i}$.
\end{rem}

In section \ref{background}, we have seen that the differential Galois group $G_N$ of the normal (variational) equation $N$  is a quotient of the differential Galois group $G$ of the whole system. If $\mathfrak{g}$ is abelian then $\mathfrak{g}_N$  is  abelian as well. The converse may not hold because a non-abelian group may have abelian quotients. We will now give criteria to detect obstructions to abelianity by means of  the coefficients of $A$ and a reduced form $R$ of $A$ when $G^\circ$ is abelian. 

\subsection{Maximal abelian subalgebras and effective abelianity conditions} \label{maximal-abelian}

At this point,
if $\mathfrak{g}_N$ is abelian, we see that $A\in \mathfrak{h}(k^\circ)$ and $\mathfrak{g}\subset\mathfrak{h}$, where $A$  and $ \mathfrak{h}$ can be chosen in 
	table  \ref{tableA} below. 
To find out whether $\mathfrak{g}$ is abelian, we will not need to compute $\mathfrak{g}$ completely:
we only need to show there is an abelian maximal subalgebra $\mathfrak{m}\subset\mathfrak{h}$ such that $\mathfrak{g}\subset\mathfrak{m}$.
 \begin{table}[h] 
\begin{center}
 \caption{Systems after reduction of the Normal Variational Equation}
\begin{tabular}{|c|c|c|c|}
\hline
$\small\mathfrak{g}_N$ & $\{0\}$ & $\small\mathfrak{g}_a$ & $\small\mathfrak{g}_m$\\
 \hline
 & & & \\
$\tiny A$ & $\begin{array}{c}a_{12}M_1+a_{14}M_2 \\+a_{13}M_3\end{array}$ & $\begin{array}{c} a_{12}M_1+a_{14}M_2 \\+a_{13}M_3+ a_{24}M_a\end{array}$ & $\begin{array}{c}a_{12}M_1+a_{14}M_2 \\+a_{13}M_3+ a_{22}M_m\end{array}$\\& & & \\
 \hline
& & & \\
$\tiny\mathfrak{h}$ & ${\tiny span_\mathbb{C}(M_1 , M_2 , M_3)}$ & ${\tiny span_\mathbb{C}(M_a , M_1 , M_2 , M_3)}$ & ${\tiny span_\mathbb{C}(M_m , M_1 , M_2 , M_3)}$\\& & & \\
 \hline
\end{tabular}
\label{tableA}
\end{center} 
\end{table}

\begin{lem}\label{Lema4}
 Let $\mathfrak{g}$ be the Lie algebra of $Y'=AY$ and let $\mathfrak{h}$ be a Lie algebra such that $\mathfrak{g}\subset\mathfrak{h}$
 and $A\in \mathfrak{h}(k^\circ)$. Then, $\mathfrak{g}$ is abelian if and only if there is some $P\in GL_{n}(k^\circ)$ such that $P[A]\in\mathfrak{m}(k^\circ)$, where $\mathfrak{m}$ is a maximal abelian subalgebra of $\mathfrak{h}$.

\end{lem}

\begin{proof}
Suppose $\mathfrak{g}$ is abelian. 
Then there will exist
some maximal abelian subalgebra  $\mathfrak{m}$ such that $\mathfrak{g}\subseteq \mathfrak{m}$. 
By Theorem \ref{Kovacic}, we know that there exists $P\in GL_n (k^\circ)$ such that
$P[A]\in\mathfrak{g}(k^\circ)\subseteq \mathfrak{m}(k^\circ)$, and we are done.
The converse follows from Theorem \ref{Kovacic} and its corollary.
\end{proof}

\begin{lem}
Keeping the notations from table  \ref{tableA}, 
we have:
\begin{enumerate}
\item if $\mathfrak{g}_N = \{0\}$, the maximal abelian subalgebras of $\mathfrak{h}$ are of the form $\mathrm{span}_{\mathbb{C}}(\alpha_1 M_1 + \alpha_2 M_2\,,\, M_3)$ where $(\alpha_1\,,\,\alpha_2)\in\mathbb{C}^2$.
\item if $\mathfrak{g}_N = \mathfrak{g}_a$, the maximal subalgebras are $\mathrm{span}_{\mathbb{C}}(M_2\,,\,M_3\,,\, M_a)$ and $\mathrm{span}_{\mathbb{C}}(M_a+\alpha_1 M_1 + \alpha_2 M_2\,,\, M_3)$ with $(\alpha_1\,,\,\alpha_2)\in \mathbb{C}^2$.
\item If $\mathfrak{g}_N = \mathfrak{g}_m$, the maximal abelian subalgebras of $\mathfrak{h}$ are of the form $\mathrm{span}_{\mathbb{C}}(M_m+\alpha_1 M_1 + \alpha_2 M_2\,,\, M_3)$ with $(\alpha_1\,,\,\alpha_2)\in \mathbb{C}^2$.
\end{enumerate}

\end{lem}

\begin{proof}

The results in the Lemma are easily deduced from the multiplication tables given below

\begin{center}
\begin{tabular}{|c|c|}\hline
[Case $\mathfrak{g}_N =\mathfrak{g}_a$  ]  & [Case $\mathfrak{g}_N =\mathfrak{g}_m$ ] \\\hline

 \begin{tabular}{|c||c|c|c|c|}\hline
$\small [\, ,\,]$&$ \small M_a$  &$ \small M_1$ & $\small M_2$ & $\small M_3$\\ \hline\hline
$\small M_a$ &$\small  0 $ & $\small -M_2 $ & $\small 0$ & $\small 0$\\\hline
$\small M_1$ &$\small  M_2 $ &$\small 0 $ & $\small 2 M_3$ & $\small 0$\\\hline
$\small M_2$ &$\small 0$ &$\small -2M_3 $ & $\small 0$ & $\small 0$\\\hline
$\small M_3$ &$\small 0 $ &$\small 0 $ & $\small 0$ & $\small 0$\\\hline
\end{tabular}

&

 \begin{tabular}{|c||c|c|c|c|}\hline
$\small [\, ,\,]$&$\small M_m$  &$\small M_1$ & $\small M_2$ & $\small M_3$\\ \hline\hline
$\small M_m$ &$\small 0 $ & $\small -M_1$ & $\small M_2$ & $\small 0$\\\hline
$\small M_1$ &$\small M_1$ &$\small 0 $ & $\small 2 M_3$ & $\small 0$\\\hline
$\small M_2$ &$\small -M_2$ &$\small -2M_3 $ & $\small 0$ & $\small 0$\\\hline
$\small M_3$ &$\small 0 $ &$\small 0 $ & $\small 0$ & $\small 0$\\\hline
\end{tabular}
\\
\hline
\end{tabular}
\end{center}

\begin{enumerate}
\item If $\mathfrak{g}_N = 0$, in the table  we see  that $[M_1\,,\, M_2]=2 M_3 \neq 0$  whereas $[M_i \,,\, M_3] =0$ for $i\in\lbrace1,2\rbrace$. Therefore, any  abelian subalgebra of $\mathfrak{h}$ must necessarily be a subalgebra of an algebra of the type $\mathrm{span}_{\mathbb{C}} (\alpha_1 M_1 + \alpha_2 M_2 \,,\, M_3)$ where $(\alpha_1 \,,\, \alpha_2)\in\mathbb{C}$.

\item If $\mathfrak{g}_N = \mathfrak{g}_a$, in the table we see that $[M_1\,,\, M_2]=2 M_3 \neq 0$  and $[M_1\,,\, M_a] =-M_2\neq 0$ whereas $[M_i \,,\, M_3] =0$ for $i\in\lbrace1,2\rbrace$. Therefore, any  abelian subalgebra of $\mathfrak{h}$ containing $M_a$ must necessarily be a subalgebra of

\begin{itemize}
\item[-] either  an algebra of the type $\mathrm{span}_{\mathbb{C}} (M_a +\alpha_1 M_1 + \alpha_2 M_2 \,,\, M_3)$ where $(\alpha_1 \,,\, \alpha_2)\in\mathbb{C}$.
\item[-] or  $\mathrm{span}_{\mathbb{C}} (M_a\, , \,M_2 \,,\, M_3)$ where $(\alpha_1 \,,\, \alpha_2)\in\mathbb{C}$.
\end{itemize}
\item If $\mathfrak{g}_N = \mathfrak{g}_m$: since the algebra $\mathrm{span}_{\mathbb{C}}(M_m\,,\, M_1 \,,\, M_2)$ is not abelian and has no abelian subalgebras 
(other than the monogenous ones),
all abelian subalgebras of $\mathfrak{h}$ containing $M_m$ must be of the type $\mathrm{span}_{\mathbb{C}}(M_m + \alpha_1 M_1 + \alpha_2 M_2 \,,\, M_3)$ where $(\alpha_1 \,,\, \alpha_2)\in\mathbb{C}^{2}$.
\end{enumerate}
\end{proof}

\begin{rem}\label{conjugate-algebras}
In the above lemma, two cases may be simplified. In the second part of case $(2)$, $\mathrm{span}_{\mathbb{C}}(M_a+\alpha_1 M_1 + \alpha_2 M_2\,,\, M_3)$  is conjugate to $\mathrm{span}_{\mathbb{C}}(M_a+\alpha_1 M_1 \,,\, M_3)$  (via $P=\mathrm{Id}+\alpha_2 M_1$). 
In case $(3)$, $\mathrm{span}_{\mathbb{C}}(M_m+\alpha_1 M_1 + \alpha_2 M_2\,,\, M_3)$ is conjugate to 
$\mathrm{span}_{\mathbb{C}}(M_m\,,\, M_3)$ (via $P=\mathrm{Id}+\alpha_1 M_1 - \alpha_2 M_2$).
\end{rem}

Before we proceed to our main result, we need a lemma about the structure of solution matrices.

\begin{lem}\label{tableU}
Let $Y'=AY$ with $A\in\mathfrak{sp}(4,k^\circ)$ as given in 
	table \ref{tableA}  page \pageref{tableA}. 
Assume that $\mathfrak{g}_N$ is abelian.
Then, depending  on $\mathfrak{g}_N$, a fundamental (symplectic) solution matrix $U\in Sp(4,K)$ can be chosen in the table below:
 \begin{center}
   \begin{tabular}{|c|c|c|c|}
    \hline
    $\mathfrak{g}_N $ & $\mathfrak{g}_{Id}=\{0\}$ & $ \mathfrak{g}_a $ & $\mathfrak{g}_m $ \\
     \hline
      & & & \\
      $U$ &
      $\small\left[\begin{array}{cccc}
      		1 & \Omega_1 & \Omega_3 & \Omega_2 \\
		0 & 1 & \Omega_2 & 0 \\
		0 & 0 & 1 & 0\\
		0 & 0 &  -\Omega_1 & 1
		\end{array}\right]$
	&
	$\small\left[\begin{array}{cccc}
		1 & \Omega_1 & \Omega_3 & \Omega_2+L  \Omega_1\\
		0 & 1 & \Omega_2 & L \\
		0 & 0 & 1 & 0\\
		0 & 0 &  -\Omega_1 & 1
		\end{array}\right] $
	&
	$\small\left[\begin{array}{cccc}
		1 &  E \Omega_1 & \Omega_3 & \frac{\Omega_2}{E}\\
		0 & E & \Omega_2 & 0 \\
		0 & 0 & 1 & 0\\
		0 & 0 & -\Omega_1 & \frac{1}{E}
		\end{array}\right] $\\
         & &\small $L:=\int a_{24}\notin k^\circ$ & \small $ E:=\exp(\int a_{22})\notin k^\circ$\\
         & & with $a_{24}\in k^\circ$ & with $a_{22} \in k^\circ$\\
  \hline
           \end{tabular}
 \end{center}

\end{lem}
\begin{proof}
This table is a direct consequence of Proposition \ref{FMS} given in the appendix (page \pageref{GSS}).
\end{proof}

We now state our main result. This next theorem gives a reduction algorithm which puts, if possible, $A$ into reduced form. In what follows, the term
"P is a reduction matrix" means that $Lie(P[A])$ is abelian (so, strictly, the system may be only partially reduced but this is enough for us).

\begin{thm}\label{algebresmaximales}
Consider a differential system $Y'=AY$, where $A$ is chosen in table \ref{tableA} page \pageref{tableA}. 
Let $\mathfrak{g}$ be the Lie algebra of its differential Galois group.
Then, $\mathfrak{g}$  is abelian if and only if
\begin{description}
\item[1. Case  $\mathfrak{g}_N = \{0\}$:]  one of the two assertions below holds.
\begin{enumerate}
\item There exist $y_{1},y_{2}\in k^\circ$ such that $y_{1}'=a_{12}$ and $y_{2}'=a_{14}$.
	In that case $\mathfrak{g}\subset \mathrm{span}_{\mathbb{C}}(M_3)$ and $P:=\mathrm{Id} + y_{1}M_{1} + y_{2}M_{2}$ is a reduction matrix.
\item
There exist $(\alpha_1\,,\,\alpha_2)\in\mathbb{C}^2 \setminus(0,0)$
 such that the equation
$ y_{1}'=\alpha_2 a_{12} - \alpha_1 a_{14}$
  has a solution $y_{1}\in k^\circ$. In that case $\mathfrak{g}\subset \mathrm{span}_{\mathbb{C}}(\alpha_2 M_1+\alpha_1 M_2 \,,\,M_3)$ 
  and $P:=\mathrm{Id} +\frac{1}{2\alpha_1\alpha_2} y_1 \left(\alpha_{1}M_{1} -\alpha_{2}M_{2}\right)$ is a reduction matrix.
\end{enumerate}

\item[2. Case  $\mathfrak{g}_N = \mathfrak{g}_a$:] one of the two assertions below holds.

	\begin{enumerate}
	\item There exists $y_{1}\in k^\circ$ such that $y_{1}'=a_{12}$. In that case $\mathfrak{g}\subset\mathrm{span}_{\mathbb{C}}(M_2,\,M_3\,,\, M_a)$ and $P:=\mathrm{Id}+ y_1 M_1$ is a reduction matrix.
	\item There exist $(\alpha_1\,,\,\alpha_2)\in\mathbb{C}^2\setminus (0,0)$ and $y_1 , y_2 \in k^\circ$ such that
			\begin{equation}\label{L2}
					\left\{\begin{array}{ccccccc}
					y'_1 & =& a_{12}& -& \alpha_1 a_{24}&&\\
					y'_2 &=& a_{14} & - & a_{24}y_1 &-& \alpha_2 a_{24}.
					\end{array}\right.
			\end{equation}
			In that case, $\mathfrak{g}\subset \mathrm{span}_{\mathbb{C}}(M_a + \alpha_1 M_1 +\alpha_2 M_2\,,\, M_3)$
			and $P:=\mathrm{Id}+ y_1 M_1+ y_2 M_2$ is a reduction matrix.
\end{enumerate}
\item[3. Case  $\mathfrak{g}_N = \mathfrak{g}_m$:]  the system of Risch equations
	\begin{equation}\label{Risch2}
		\left\{ \begin{array}{ccccc}
					 y'_1 &=&  -a_{22}y_1 & + &  a_{12}\\
					 y'_2 &=& a_{22}y_2 &+&  a_{14}
			\end{array}  \right.
 	\end{equation}

has a non-trivial solution in $(k^\circ)^2$. In that case $\mathfrak{g}\subset\mathrm{span}_{\mathbb{C}}(M_m \,,\,M_3)$ and 
$P:=\mathrm{Id} + y_1  M_1 + y_2 M_2$ is a reduction matrix.
\end{description}
\end{thm}

\begin{proof}(Theorem)
We first show the case $\mathfrak{g}_N = \{0\}$ in detail, together with a reduction procedure. The next cases will then
be easier to follow.
\begin{description}
\item[1. Case $\mathfrak{g}_N = 0$:]
In this case (by lemma  \ref{tableU}), we have $A=a_{12} M_1 + a_{14} M_2 + a_{13}M_3$.
Lemma \ref{tableU} shows that the system $Y'=AY$ has a fundamental solution matrix of the form
$$U= \mathrm{Id} + \Omega_1 M_1 + \Omega_2 M_2 + \Omega_3 M_3$$

Assume now that $\mathfrak{g}$ is abelian. By Lemma \ref{Lema4}  This is equivalent to
$$\mathfrak{g}\subset \mathfrak{m}_{\alpha_{1},\alpha_{2}} := \mathrm{span}_{\mathbb{C}}(\alpha_1 M_1 + \alpha_2 M_2\,,\,M_3)
\; \hbox{\rm with }\; (\alpha_1\,,\,\alpha_2)\in\mathbb{C}^2.$$
We pick $\sigma\in \exp(\mathfrak{g} ) \subset G^\circ$ and study its action on $U$.
Using relations (\ref{produitMi}) and abelianity of $\mathfrak{m}_{\alpha_{1},\alpha_{2}}$, we see that
$\sigma=\mathrm{Id} + \beta_1(\alpha_1 M_1 + \alpha_2 M_2) + \beta_3 M_3$ for some $\beta_{1},\beta_{3}\in \mathbb{C}$.

By a quick calculation using relations (\ref{produitMi})) the action of $\sigma$ over $U$ is 
 given by $ \nonumber\sigma(U)=U.\sigma = $
 $$
 \mathrm{Id} + (\Omega_1 + \beta_1\alpha_1)M_1 +
 	(\Omega_2 + \beta_1 \alpha_2)M_2 +
(\Omega_3 + \beta_3 + \beta_1(\alpha_2 \Omega_1  - \alpha_1\Omega_2 ))M_3.
$$
Identifying this with $\sigma$ applied to coefficients of $U$, we infer that
$$\left\{ \begin{array}{ccc}
	\sigma(\Omega_{1}) & = & \Omega_{1}+\alpha_{1}\beta_{1} \\
	\sigma(\Omega_{2}) & = & \Omega_{2}+\alpha_{2}\beta_{1}
  \end{array}\right .
 $$
 
 If $\alpha_{1}=\alpha_{2}=0$, this shows that $\Omega_{1}$ and $\Omega_{2}$ are fixed by
 $\exp(\mathfrak{g})$ hence by $G^\circ$ (e.g \cite{FuHa91a}, proof of proposition 8.33) so they are algebraic. Now we also have
 $\Omega_{i}'\in k^\circ$ hence $\Omega_{i}\in k^\circ$ (for $i=1,2$).\\
 We introduce the reduction matrix $P:=\mathrm{Id} + \Omega_{1}M_{1} + \Omega_{2}M_{2} \in GL_{4}(k^\circ)$
 and easily check that $U=P.\left(\mathrm{Id} +\Omega_{3}M_{3}\right)$.
 Hence, 
the new system has matrix $P[A]=(a_{13} + a_{12} \Omega_2 - a_{14} \Omega_1 ) M_{3}$ and
 	$\mathfrak{g}\subset \mathrm{span}_{\mathbb{C}}(M_3)$.
	
If now $(\alpha_{1},\alpha_{2})\neq (0,0)$, we let $f:=\alpha_{2}\Omega_{1}-\alpha_{1}\Omega_{2}$. It is left fixed
by $G^\circ$ and $f'=\alpha_{2} a_{12} -\alpha_{1} a_{14} \in k^\circ$ so $f \in k^\circ$ as above.
We use the reduction matrix
$P:=\mathrm{Id} +\frac{1}{2\alpha_1\alpha_2} f \left(\alpha_{1}M_{1} -\alpha_{2}M_{2}\right) \in GL_{4}(k^\circ)$
and obtain $$P[A]=\beta\left( \alpha_1 M_1 + \alpha_2 M_2\right) + (a_{13}-2 f \beta) M_3 
\in \mathfrak{m}_{\alpha_{1},\alpha_{2}}(k^\circ)
\quad \hbox{\rm where}\quad \beta = \frac{\alpha_2 a_{12} + \alpha_1 a_{14}}{2\alpha_1 \alpha_2} $$
 so $\mathfrak{g}\subset \mathfrak{m}_{\alpha_{1},\alpha_{2}}$ and the system is again (partially) reduced with an abelian Lie algebra.

Conversely, if one of the conditions of the theorem holds, the reduction (gauge transformation) matrices $P$ given above
send $A$ to a $k^\circ$-point of $\mathrm{span}_{\mathbb{C}}(M_3)$ or $\mathfrak{m}_{\alpha_{1},\alpha_{2}}$
respectively, showing that $\mathfrak{g}$ is abelian.

\item[2. Cases $\mathfrak{g}_N = \mathfrak{g}_a$:]
In this case $\mathfrak{g}$ must always contain $M_a$ otherwise we would be in the case $\mathfrak{g}_n = 0$. By lemma \ref{tableU}, we have
 $A=a_{12} M_1 + a_{14}M_2 + a_{13} M_3 + a_{24} M_a$. The above lemma shows that the system $Y'=AY$ has a fundamental matrix of the form
$$U=\mathrm{Id}+\Omega_1 M_1 + \Omega_2 M_2 +\Omega_3 M_3 + L M_a -L\Omega_1 M_1 M_a .$$
Assume now that $\mathfrak{g}$ is abelian. This is equivalent (Lemma \ref{Lema4})  to either
$\mathfrak{g}\subset \mathfrak{m}:=\mathrm{span}_\mathbb{C}(M_2\,,\, M_3\,,\, M_a)$ or  $\mathfrak{g}\subset \mathfrak{m}_{\alpha_1 ,\alpha_2}:=\mathrm{span}_\mathbb{C}(M_a + \alpha_1 M_2 + \alpha_2 M_1\,,\, M_3\,)$:
\\
\begin{itemize}
\item[(a)] $\mathfrak{g}\subset \mathfrak{m}:=\mathrm{span}_\mathbb{C}(M_2\,,\, M_3\,,\, M_a)$:  pick $\sigma\in \exp(\mathfrak{g} ) \subset G^\circ$ and study its action on $U$.
Using relations (\ref{produitMi}) and abelianity of $\mathfrak{m}$, we see that
	$\sigma=\mathrm{Id} + \beta_2 M_2 + \beta_3 M_3 + \beta M_a$ for some $\beta_{2},\beta_{3},\beta \in \mathbb{C}$. Computing the action of $\sigma$ over $U$  we obtain $U\sigma = $
$$\mathrm{Id} + \Omega_1 M_1 + (\beta_2 + \beta_1 \Omega_1 + \Omega_3)M_3 + (\Omega_2 + \beta_2)M_2 + (\beta \Omega_1 - L\Omega_1)M_1 M_a + (\beta + L) M_a$$
Identifying this with $\sigma$ applied to coefficients of $U$, we infer that  $\sigma(\Omega_1)=\Omega_1$ which is equivalent to
	$\int a_{12}=:\Omega_1 \in k^\circ$. The reduction matrix obtained this way  is
$$P=Id + \Omega_1 M_1$$
and 
 $P[A] = (a_{13} - 2\, a_{14}\,\Omega_1 + a_{24}\,\Omega^{2}_1)M_1 + (a_{14}-a_{24}\,\Omega_1 )M_2 + a_{24}\,M_{a}$.
\\
\item[(b)] $\mathfrak{g}\subset \mathfrak{m}:=\mathrm{span}_{\mathbb{C}}( \alpha_1 M_1 + \alpha_2 M_2 + M_a\,,\,M_3)$:
 pick $\sigma\in \exp(\mathfrak{g} ) \subset G^\circ$ with
 	$\sigma = \exp(\beta_1(M_a +\alpha_1 M_1 + \alpha_2 M_2) + \beta_3 M_3 )$. We compute $U.\sigma$ and compare this with entries of $\sigma(U)$
	to obtain:
$$\left\{ \begin{array}{ccc}
	\sigma(\Omega_{1}) & = & \Omega_{1}+\alpha_{1}\beta_{1} \\
	\sigma(L)&=& L+\beta_1\\
	\sigma(\Omega_{2}) & = & \Omega_{2}+\alpha_{2}\beta_{1}- L\beta_1 \alpha_1 -\frac{\beta^2_1 \alpha_1}{2}
  \end{array}\right.
 $$
Let $y_{1}:= \Omega_{1}-\alpha_{1}L$; the first two relations show that $y_{1}$ is fixed by all $\sigma$. Hence
	$y_{1}\in k^\circ$ and $y_{1}$ is a solution in $k^\circ$ of $y_{1}'=a_{12}-\alpha_{1}a_{24}$.
Similarly (but with a few lines of calculation), letting $y_{2}:=\Omega_{2}-\alpha_{2}L+\frac12 \alpha_{1} L^2$
	(i.e $y_{2}'=a_{14}-a_{24}y_{1}-\alpha_{2}a_{24}$), the third condition shows that $y_{2}\in k^\circ$
The reduction matrix obtained this way  is
$$P_{\mathfrak{m}}=\mathrm{Id} + y_1 M_1 + y_2 M_2\quad\text{ and }\quad P_{\mathfrak{m}}[A]= a_{24}(M_a + \alpha_1 M_2 + \alpha_2 M_1)+h M_3 $$
with $h = y_{2}' y_{1} - y_{1}' y_{2} + a_{13}$.
\end{itemize}

\item[3. Case  $\mathfrak{g}_N=\mathfrak{g}_m$]:
Again, $M_m$ must belong to $\mathfrak{g}$ (otherwise we would be in the case $\mathfrak{g}_N = 0$).
 By Lemma \ref{tableU} we have $A=a_{12} M_1 + a_{14} M_2 + a_{13} M_3 + a_{22} M_m$.
 Lemma \ref{tableU} shows that the system $Y'=AY$ has a fundamental matrix of the form
 $$
 U:=\left[\begin{array}{cccc} 1 & E\Omega_1  & \Omega_3 & \frac{\Omega_2}{E}\\
 0& E & \Omega_2 & 0 \\
 0 & 0 & 1 & 0\\
 0 & 0 & -\Omega_1 & \frac{1}{E}\end{array}\right].
 $$

Assume now that $\mathfrak{g}$ is abelian:  by Lemma \ref{Lema4}, this is equivalent to $\mathfrak{g}\subset \mathrm{span}_\mathbb{C}(\alpha_1 M_1 + \alpha_2 M_2 +M_m, M_3)$ for some $\alpha_1 ,\alpha_2 \in\mathbb{C}$.
We pick a $\sigma\in \exp(\mathfrak{m})\subset G^{\circ}$: notice that  since $M^{2}_m\neq 0$  the form of $\sigma$ will be
$$
M_{\sigma}:=\left[\begin{array}{cccc} 1 & \alpha_1 (\mathrm{e}^{c} -1) & \alpha_1 \,\alpha_2 \,(\mathrm{e}^{c} -\mathrm{e}^{-c} -2c) + \beta & \alpha_2 \left(1-\mathrm{e}^{-c}\right)\\
0 & \mathrm{e}^{c} & \alpha_2 (\mathrm{e}^{c}- 1 )  & 0 \\
0 & 0 & 1 & 0\\
0 & 0 & -\alpha_1 (\mathrm{e}^{c} -1)  & \mathrm{e}^{-c}\end{array}\right]
$$
with $ c \in\mathbb{C}^{\star}, \beta,\alpha_i \in\mathbb{C}$ for $i=1,2$.  Computing $\sigma(U)=U.M_{\sigma}$ we obtain the following relations linking  $E, \Omega_1 ,\Omega_2$ and the constants linked to $\sigma$:
 $$\left\{ \begin{array}{ccc}
 		\sigma(E) &=&\mathrm{e}^{c} E \\
	\sigma(\Omega_{1}) & = & \Omega_{1} - \frac{\alpha_{1}}{E}+\frac{\alpha_{1}}{E \mathrm{e}^{c}} \\
	\sigma(\Omega_{2}) & = & \Omega_{2}+\alpha_{2}\mathrm{e}^{c} E     - E \alpha_2.
  \end{array}\right.
 $$
Again, it is easily checked that $\Omega_{1}-\frac{\alpha_{1}}{E}$ and $\Omega_{2} -\alpha_{2} E$ are fixed by all $\sigma$.
This is equivalent with saying that the following equations
  $$\left\{ \begin{array}{ccc}
y'_1 & = & -a_{22} y_1  + a_{12}  \\
y'_2 & = & a_{22} y_2  +a_{14}
  \end{array}\right.
  $$
have solutions $y_{1},y_{2} \in k^\circ$.
The reduction matrix obtained this way  is
$$P_{\mathfrak{m}}=\mathrm{Id} + y_1  M_1 + y_2 M_2  \quad\text{ and }\quad P_{\mathfrak{m}}[A]= (a_{12} y_2+a_{13}+a_{14}y_1)M_3 + a_{22}M_m.$$

\end{description}
\end{proof}

\begin{rem}
In the case $\mathfrak{g}_N = \mathfrak{g}_m$, we actually proved that $\mathfrak{g}$ is a subalgebra of $\mathrm{span}_{\mathbb{C}}(M_3\,,\, M_m)$ instead of the maximal abelian subalgebra  $\mathfrak{m} =\mathrm{span}_{\mathbb{C}}(\alpha_1 M_1 + \alpha_2 M_2 + M_m\,,\, M_3)$. Indeed, if we pick a differential system $[B]\:\,Y' = BY$ with $B:=f(x)\,(\alpha_1 M_1 + \alpha_2 M_2 + M_m) + g(x) M_3\in \mathfrak{m}(k)$ the abelianity conditions given in the theorem are written as:
$$
\left\lbrace
\begin{array}{ccccccc}
y'_1 &=&-f(x) (&y_1& -& \alpha_1&)\\
y'_2 &=& \,\, f(x) (&y_2& +& \alpha_2&)
\end{array}\right. 
$$
and this system  always has a non-trivial (constant) rational solution $(y_1\,,\, y_2) = (\alpha_1 \,,\, -\alpha_2)$. Therefore, there exists a linear change of variables given by $P_{\mathfrak{m}} := \mathrm{Id} + \alpha_1 M_1 - \alpha_2 M_2$ such that $P_{\mathfrak{m}}[B]=g(x) M_3 + f(x) M_m$. This recovers what we saw in remark \ref{conjugate-algebras}, the fact that these two algebras are conjugate. A similar thing (using again remark \ref{conjugate-algebras}) holds for the second additive case.
\end{rem}

\subsection{A reduction algorithm}\label{A first Reduction Algorithm}

We start with a partial reduction algorithm which summarizes the procedures established above.
\begin{description}
 \item[INPUT:] A matrix $\mathbf{A}\in \mathfrak{sp}(4,k)$, with $k$ an admissible differential field,   and a particular solution  of  $ Y' = \mathbf{A}Y$.
\begin{description}
\item[Step 1:] Perform   symplectic  reduction using section 2.
\item[Step 2:] Reduce the $NVE$ applying Kovacic's algorithm as in section 2 (gives the reduction field $k^\circ$).
\item[Step 3:] If $Lie(NVE)$ is abelian, follow the algorithm of Theorem \ref{algebresmaximales}
\end{description}
 \item[OUTPUT:] either "$\mathfrak{g}$ non-abelian" or $P\in \mathrm{GL}(4\,,\, k^\circ)$ such that $Lie(P[A])$ is abelian
\end{description}
\vspace{0.2cm}


We call $B:=P[A]$ the output of this algorithm and let $\mathfrak{m}:=Lie(B)$.
We could now wish to complete the reduction in the case when $\mathfrak{g}\subsetneq \mathfrak{m}$.
The Lie algebra $\mathfrak{g}$ is a subalgebra of $\mathfrak{m}$ but not \textit{any} algebra. Indeed, when $\mathfrak{g}_N$ is $\mathfrak{g}_a$  (resp. $\mathfrak{g}_m$), the matrix $M_a$ (resp. $M_m$) must belong to $\mathfrak{g}$. Otherwise we would fall again in the case $\mathfrak{g}_N =0$.
We give in the next lemma the list of possible $\mathfrak{g}$.

\begin{lem}\label{sousalgebres}
List of all possible abelian subalgebras:
\begin{enumerate}
\item[1.] The subalgebras of $\mathfrak{m}=\mathrm{span}_{\mathbb{C}}(\alpha_1 M_1 + \alpha_2 M_2\,,\, M_3)$ are among 
	$\{0\}$, $\mathfrak{m}$, of the form $\mathrm{span}_{\mathbb{C}}(\alpha_1 M_1 + \alpha_2 M_2)$  with $(\alpha_1 \,,\, \alpha_2)\in\mathbb{C}^{2} \setminus (0,0) $, or $\mathrm{span}_{\mathbb{C}}(M_3)$.

\item[2.] The subalgebras of $\mathfrak{m}=\mathrm{span}_{\mathbb{C}}(M_a +\alpha_1 M_1 +\alpha_2 M_2\,,\, M_3)$ that contain $M_a$ are among 
	$\{0\}$, $\mathfrak{m}$, 
	or of the form $\mathrm{span}_{\mathbb{C}}(\alpha_1 M_1 +\alpha_2 M_2 + \alpha_3 M_3 + M_a)$ with $(\alpha_1 \,,\, \alpha_2\,,\, \alpha_3)\in\mathbb{C}^{3}$.
	
\item[3.] The subalgebras of $\mathfrak{m}=\mathrm{span}_{\mathbb{C}}(M_2\,,\, M_3\,,\, M_a)$ that contain $M_a$  are among 
	$\{0\}$, $\mathfrak{m}$, 
of the form $\mathrm{span}_{\mathbb{C}}(\alpha_2 M_2 + \alpha_3 M_3\,,\, M_a)$ with $(\alpha_2 \,,\, \alpha_3)\in\mathbb{C}^{2}$,
or of the form $\mathrm{span}_{\mathbb{C}}(M_2 \,,\ \alpha_3 M_3 + M_a)$ with $\alpha_3 \in \mathbb{C}^{\star}$,
or of the form $\mathrm{span}_{\mathbb{C}}(M_3 \,,\ \alpha_2 M_2 + M_a)$ with $\alpha_2 \in \mathbb{C}^{\star}$,
or $\mathrm{span}_{\mathbb{C}}(\alpha_2 M_2 + \alpha_3 M_3 + M_a)$ with $(\alpha_2 \,,\, \alpha_3)\in\mathbb{C}^{2}$.

\item[4.] The subalgebras of $\mathfrak{m}=\mathrm{span}_{\mathbb{C}}(M_m \,,\, M_3)$ that contain $M_m$  are among 
	$\{0\}$, $\mathfrak{m}$, 
or of the form  $\mathrm{span}_{\mathbb{C}}(M_m + \alpha_3 M_3)$ with $\alpha_3\in\mathbb{C}$.

\end{enumerate}
\end{lem}
\begin{proof}
We leave the computation to the reader.
\end{proof}

 We may infer a reduction algorithm from the above lemma. Namely, the proofs of propositions 5, 6, 7 and 8 in chapter 4 of  \cite{Ap10a} (one proposition per numeric item of Lemma \ref{sousalgebres}) sum up a reduction algorithm.
 Furthermore, the reduction will be maximal because the procedure proposed deals only with abelian Lie algebras.
 An outline of the proof of those four proposition runs as follows.  We pick for each case an element $\sigma \in \mathrm{exp}(\mathfrak{g})$ and compute its action over a symplectic fundamental matrix of solutions $U$. From the relations obtained by means of these computations we extract the reducibility conditions. The reduction matrix and corresponding reduced form are obtained using a direct computation.   
 To prove that $\mathbf{P}[B]$ is reduced instead of just partially reduced, 
 we prove either that  $\mathfrak{g}$ is monogenous or that $\mathrm{dim}_{\mathbb{C}}(\mathfrak{g})$ is equal to the minimal number of generators of $\mathfrak{g}$ as a Lie algebra.  At the outcome, we obtain a fully reduced form for our differential system.

\begin{rem}
When a system is in reduced form, one may sometimes yet find a simpler reduced form. For example, 
  assume we were able to write some coefficients as $a=f'+g$ with $f,g\in k^\circ$; 
  for instance, if $k=\mathbb{C}(x)$ we can use partial fraction decomposition and Hermite reduction to obtain $a=f'+g$ (and $g$ with simple poles).
Then, like in the reduction process, we can find a gauge transformation which will remove the $f'$ part of $a$, leaving only $g$.
\end{rem}

\begin{exmp}
Consider the system $Y'=MY$ given by $ M=\frac{1+x^2}{x} M_1 + \frac{x^3 - 4x^2 +1}{x-4} M_3\in\mathfrak{sp}(2,\mathbb{C}(x))$ with $f_1(x)=\frac{1+x^2}{x}$ and $f_3(x)=\frac{x^3- 4x^2 +1}{x-4} $. 
Its Lie algebra is $\mathfrak{g}_M = \mathrm{span}_{\mathbb{C}}(M_1\,,\, M_3)$. It is easily checked that $M$ is already in a reduced form. Using partial fraction decomposition, we obtain: $ f_1(x) = x + \frac{1}{x}$ and $\quad f_3(x) = x^2 + \frac{1}{x-4}$. Now, applying the gauge transformation $Q:=\mathrm{Id} + \frac{x^2}{2}M_1 + \frac{x^3}{3}M_3 \in \mathrm{GL}_{4}(\mathbb{C}(x))$ we obtain the equivalent and simpler system $Q[M]=\frac{1}{x}M_1 + \frac{1}{x-4}M_3$.
\end{exmp}
   This kind of procedure can be generalized to any monomial and elementary extensions of $k$ where $k$
   is a differential field of characteristic $0$, see \cite{Br05a}.
\section{Reduction and integrability for Hamiltonian systems in dimension $n=4$ }
\subsection{The algorithm}

The reduction algorithm shown in the previous section is easily turned into an effective version of the Morales-Ramis criterion for meromorphic Hamiltonian systems with two degrees of freedom.
 Indeed, consider a Hamiltonian system such as (\ref{Sham}) with two degrees of freedom and assume that $\Gamma$ is one of its non-constant integral curves.
  The following algorithm either returns a reduced form of the variational equation of (\ref{Sham}) along $\Gamma$, or tells us that the system (\ref{Sham}) is not meromorphically Liouville integrable.\\
\\
\textbf{INPUT :} A meromorphic Hamiltonian function $H:\, U\subset \mathbb{C}^{2n}\longrightarrow \mathbb{C}$, a particular non-stationary integral curve $z(t)$ for the system (\ref{Sham}).
\begin{itemize}
\item[1.] Apply the abelianity test
to ${A}:=Hess(H)(z(t))$.
\item[2.] If $\mathfrak{g}$ is abelian apply the reduction algorithm otherwise stop.
\end{itemize}
\textbf{OUTPUT :} If $\mathfrak{g}$ is abelian we obtain a reduced form ${R}\in\mathfrak{g}(k^\circ)$. Otherwise $\mathfrak{g}$ is not abelian and the Hamiltonian system  (\ref{Sham})  is not meromorphically Liouville integrable.

\subsection{An application: the lunar Hill Problem}\label{Hill}

Let us now apply our algorithm to a real problem: a simplification of the Sun-Earth-Moon problem, more widely known as the Hill Problem. Meromorphic non-integrability of this system has been proved by Morales, Ramis and Sim\'o  in \cite{MoSiSi05a}. Applying our techniques, we give a new (simpler) proof of this result. Although, the lunar Hill problem  is modeled by a Hamiltonian function with three degrees of freedom like all three bodies problems, it is shown in \cite{MoSiSi05a} that via a symplectic change of variable over $\bar{k}$, the lunar Hill Problem can be seen as the two degrees of freedom Hamiltonian system given by the Hamiltonian function:
$$H:=\mathbf{i}(q_1 q_2 - p_1 p _2)-4q_1 q_2 (q_1 p_1 - q_2 p_2) - 4\mathbf{i}(3q^4_1 -2 q^2_1 q^2_2 + 3q^2_2)q_1 q_2.$$
We compute $X_H$ (see \cite{MoSiSi05a}) and notice that it possesses two invariant planes: 
$$\Pi_1 := \left\lbrace q_2 = 0\, , \, p_1 = 0\right\rbrace \text{ and }\Pi_2 := \left\lbrace q_1 = 0\, , \, p_2 = 0\right\rbrace. $$ 
One checks easily that $X_H$ has a particular solution over the invariant plane $\Pi_1$ given by $\gamma(x) = (f(x) ,  0 , 0 , \mathbf{i}f'(x))$ such that $f(x)$ satisfies
$ (f'(x))^2 = - f(x)^2 + 4 f(x)^6 + 2h$ so that, by differentiating, $f''(x) = -f(x) + 12 f(x)^5$. We may take $f(x)^2 =\frac{6h}{3\wp(x)+1}$
where $h\in\mathbb{C}$ is an indeterminate constant and $\wp(x) = \wp(x;\frac{4}{3},\frac{8}{27}(1-216h^2))$ is a Weierstrass-$\wp$ function. We work  over the differential field   $k:=\mathbb{C}(f(x),f'(x))$. We refer to the elements of $k$ as rational functions. Consider now the variational system over $\Pi_1$ along $\gamma$. Its matrix ${A}$ is
$$
{A} :=\left[\begin{array}{cccc} 0 & -4f(x)^2 & 0 & -\mathbf{i}\\
								 0 & 0 & -\mathbf{i} & 0\\
								 0 & -\mathbf{i}(1-60f(x)^4)& 0 & 0\\
								 -\mathbf{i}(1-60f(x)^4) & -8\mathbf{i}f(x)f'(x) & 4f(x)^2 & 0 \end{array}\right].
$$
We perform step 1. of the algorithm: we compute $ {P}$ and
	$A_N:= {P}[{A}]= {P}^{-1} ({A}  {P} -P')$:
$$
 {P}:=\left[\begin{array}{cccc} f'(x) & 0 & 0 & 0\\
								 0 & 1 & 0 & 0\\
								 0 & \frac{\mathbf{i}f''(x)}{f'(x)}& \frac{1}{f'(x)} & 0\\
								 -\mathbf{i}f''(x) & 0 & 0 & 1 \end{array}\right]
\,\Longrightarrow\, A_N:=\left[\begin{array}{cccc}
								 0 & -\frac{4f(x)^2}{f'(x)} & 0 & -\frac{\mathbf{i}}{f'(x)}\\
								 0 & \frac{f''(x)}{f'(x)} & -\frac{\mathbf{i}}{f'(x)} & 0\\
								 0 & 0 & 0 & 0\\
								 0 & F(x)  & \frac{4f(x)^2}{f'(x)} & -\frac{f''(x)}{f'(x)} \end{array}\right]
$$

\noindent where
$F(x):=\frac {8\,\mathbf{i}\, f(x)  \left(8\, f (x )^{6}-2\,h \right) }{f' (x) }$.
Thus, the normal variational equation matrix $N$ is defined by:
$$
N:=\left[\begin{array}{cc}\frac{f''(x)}{f'(x)} &  0\\
F(x)  & -\frac{f''(x)}{f'(x)}
\end{array}\right]\cdot
$$

Variation of constants shows that the Normal Variational Equation has a fundamental solution matrix
$$
\,U_N:=\left[\begin{array}{cc}f'(x) &  0\\
\frac{8i\left(f(x)^8-h f(x)^2\right)}{f'(x)}   & \frac{1}{f'(x)}
\end{array}\right].
$$
Thus we have $U_N\in Sp(2,k)$ which is equivalent (in the notations of proposition \ref{Kovacic algorithm} on page \pageref{Kovacic algorithm})
to $\mathfrak{g}_N =0$. Now reduce the normal variational equation: the matrix $P_N\in Sp(4,k)$ puts $A_N$ into the form $A:=P_N[A_N]$ where
$$
P_N := \left[\begin{array}{cccc}1 & 0 & 0& 0
\\0& f'(x) &  0 & 0
\\ 0& 0& 1 & 0
\\0& \frac{8i\left(f(x)^8-h f(x)^2\right)}{f'(x)} & 0   & -\frac{1}{f'(x)} \end{array}\right]
\quad\hbox{\rm and }
\quad
A:=\left[\begin{array}{cccc}
0 & 4G(x)& 0 & -\frac{\mathbf{i}}{f'(x)^2}\\
0 &0& -\frac{\mathbf{i}}{f'(x)^2} &0
\\ 0& 0& 0&0\\
 0 &0& -4G(x) & 0\end{array}\right].
$$
with
$ G(x) ={\frac { f(x)^{2}
 \left(-2\,f(x)^{6}-4\,h+ f(x)^{2} \right) }{4\,f(x)^{6} - f(x)^{2}+2\,h}}
\in k$.

We apply our reduction algorithm (theorem \ref{algebresmaximales} page \pageref{algebresmaximales}).
We are in case $1$ so
we search for (non-zero) constants $(\alpha,\, \beta) \in\mathbb{C}^{2}\setminus (0\,,\,0)$ such that
$\int (\alpha 4 G(x) + \beta\frac{i}{f'(x)^2}) \in k$.
Changing variable $x$ to $t=f(x)$ in $I$ (so $dx=\frac{dt}{\sqrt{4t^6-t^2+2h}}$), we are reduced to a limited integration problem on a hyperelliptic curve.
 Using section 4.1 and setting $h=1$,  this is in turn equivalent with  finding $(\alpha\,,\, \beta)\in\mathbb{C}^{2}\setminus (0,0)$ such that
\begin{equation}\label{Int}
I:=\int \alpha 4 F_1(t) + \beta F_2(t) dt \in \mathbb{C}(t,\sqrt{4t^6 - t^2 + 2})
\end{equation}
where
$$
F_1(t):=\frac{4 t^2 (-2t^6 + t^2 - 4)}{(4 t^6 - t^2 + 2)^{3/2}} \quad \text{ and }\quad
F_2(t):=\frac{ \mathbf{i}}{(4 t^6 - t^2 + 2)^{3/2}} .
$$
To solve this algorithmically, we apply our method from lemma \ref{admissible-field} page \pageref{admissible-field}.
Each  $F_i$ is a solution of $L_i := \frac{d}{dt} - \frac{F'_{i}(t)}{F_{i}(t)}$ for $i=1,2$. Hence, $\alpha F_1(t) + \beta F_2(t)$ will be a solution of $LCLM(L_1 \,,\, L_2)$  for all $(\alpha , \beta)\in\mathbb{C}^{2}$ which implies that $I$ is a solution of the differential operator $L:=LCLM(L_1\,,\, L_2) \frac{d}{dt}$. Thus, solving our limited integration problem will amount to finding non-trivial rational solutions (i.e solutions in $ \mathbb{C}(t,\sqrt{4t^6 - t^2 + 2})$ other than $1$) for the (parameter free) differential operator
$$
L:={\left(\frac{d}{dt}\right)}^{3}+{\frac { \left(44\,{t}^{6}-3\,{t}^{2}-2 \right) }{
 \left(-{t}^{2}+4\,{t}^{6}+2 \right) t}}\left(\frac{d}{dt}\right)^2+3\,{\frac {{t}^{2} \left(48\,{t}^{8}-24\,{t}^{4}+96\,{t}^{2}-1 \right) }{ \left(-{t}^{2}+4\,{t}^{6}+2 \right) ^{2}}}\frac{d}{dt}.
$$
By studying the singularities of $L$,  we prove that the latter does not have non-trivial rational solutions.
Indeed, the equation has three regular singularities:  $0$ with exponents $\{0\,,\, 1\,,\, 3\}$,
 $\alpha$ such that $4\alpha^6-\alpha^2 + 2 =0$ with exponents $\{0\,,\, 1/2\,,\, 3/2\}$ and
last $\infty$ with exponents $\{0\,,\, 0\,,\, 8\}$.
All these singularities have minimal exponent equal to $0$ so  any  solution of $L$ in $C(t,\sqrt{4t^6-t^2+2})$ must be constant (\cite{SiUl93a}).

Therefore, Theorem \ref{algebresmaximales}  shows that $\mathfrak{g}$ is not abelian which implies that the Hill Problem is not meromorphically integrable.
\begin{rem}
For the meromorphical non-integrability of the Hill problem, one of the referees suggested an alternate proof using integration on an algebraic curve instead of the language of differential operators. Let us give its outline. Taking as above $h=1$, we can write the integral $I$ as
$$I:=\int \frac{P(t)}{D(t)\sqrt{D(t)}} dt $$
 where $P(t):=4 \alpha t^2 (-2 x^6 +x^2 - 4) + \mathbf{i}\beta$ and $D(t):=4 t^6 - t^2 + 2$. Proving our point, amounts to proving that unless we pick $(\alpha\,,\, \beta) = (0\,,\,0)$, the integral $I\notin\mathbb{C}(t,\sqrt{D(t)})$. 
 We assume that if $I$ is rational then there exist some $A(t), B(t)\in\mathbb{C}(t)$ such that $I=A(t)+B(t)\sqrt{D(t)}$. Differentiating this expression, we obtain $A'(t)=0$ and therefore
 \begin{equation}
 B'(t)+\frac{D'(t)}{2D(t)} B(t)-\frac{P(t)}{D(t)^2}=0 .
 \end{equation}
Solving this Risch equation we are done. Indeed, we see that the only rational solution it admits corresponds to $(\alpha\,,\,\beta)=(0,0)$ and is trivial. It is true that this proof only requires the resolution of a Risch equation instead of handling a third order differential operator. However, we think that our first choice is a good one. Indeed, we avoid discussing parameters $(\alpha,\beta)$ and 
our method can be applied systematically to any situation satisfying the conditions stated in section 4.1.
\end{rem}
\begin{rem}
\begin{enumerate}
\item[1.] if we had chosen the energy level $h=0$, the same argument leads to looking for solutions in
$C(t,\sqrt{4t^6-t^2})$
of the differential operator
$${\left(\frac{d}{dt}\right)}^{3}+{\frac { \left(44\,{t}^{4}-3 \right)  }{t \left(2\,{t}
^{2}-1 \right)  \left(2\,{t}^{2}+1 \right) }}\left(\frac{d}{dt}\right)^2
+3\,{\frac { \left(48\,{t}^{8}-24\,{t}^{4}-1 \right) }{{t}^{2} \left(2\,{t}^{2}-1 \right) ^
{2} \left(2\,{t}^{2}+1 \right) ^{2}}} \frac{d}{dt}$$
and, indeed, we find (using {\sc Maple}), the solution $\frac {-1+8\,{t}^{4}}{\sqrt {2\,{t}^{2}+1}\sqrt {2\,{t}^{2}-1}{t}^{2}} \in k$.
So, on the energy level $h=0$, our reduction method shows that the system has an abelian Lie algebra.
\\
\item[2.] In an example like this where coefficients are parametrized by Weierstrass functions, one would need in general to use the special algorithms developed in \cite{Si91a} and improved in \cite{BuLaHo04a} to achieve the reduction.
\end{enumerate}
\end{rem}
\begin{rem}\label{NVE}
Notice that the \textit{normal variational equation} computed in our proof and the one given in \cite{MoSiSi05a} are different because they arise from two different definitions of Normal Variational equation.
On one hand, the construction we use can  be  termed as \emph{ algebraic }normal variational equation  since it is obtained by a purely algebraic manipulation (see section \ref{222}): to wit, a linear Hamiltonian change of variable obtained from a parametrization of the integral curve $\Gamma$ via the symplectic Gram-Schmidt algorithm.
On the other hand, the notion used in  \cite{MoSiSi05a}  could be qualified as \emph{geometric} normal variational equation since reduction (in general not symplectic reduction) is performed with respect to an invariant plane that contains the nondegenerate integral curve $\Gamma$. Both notions coincide in the case when the invariant plane considered in the geometric construction is actually a symplectic manifold (with symplectic form the restriction to the invariant plane of the global symplectic form).

\end{rem}

\section{Conclusion}
The notions of reduction and reduced form developed in this paper provide a procedure to decide the abelianity of the Lie algebra of the differential Galois group of the variational equations $[A]$ of Hamiltonian systems. Previously, applications of the Morales-Ramis criterion were generally limited to normal variational equations.
\\

When the Lie algebra is indeed abelian, putting the system into a reduced form is very convenient because it also allows to (partially) reduce higher variational equations in view of a concrete application of the Morales-Ramis-Sim\'o criterion.
On one hand, the higher variational equations are reducible linear differential systems whose diagonal blocks are $A$ and its symmetric powers (in the sense of Lie algebras)~; knowing a (partial) reduction matrix $P$ for $[A]$, its symmetric power $Sym^m(P)$ is a (partial) reduction matrix for $sym^m(A)$ hence inducing a partial reduction on the higher variational systems. 
Applying the techniques of section 4 to develop constructive abelianity criteria for such systems is used in \cite{Ap10a} and in \cite{AW11a} and will be the subject of other future work.\\

In many cases, Hamiltonian systems come as parametrized systems, for example with a base field $C=\mathbb{Q}(t_{1},\ldots,t_{s})$. Though there cannot exist an algorithm deciding for which values of the parameters the variational equation will admit rational solutions (\cite{Bo00a}), it turns out that in many situations, authors have been able to apply criteria like the Kovacic algorithm (or variants) to overcome that difficulty. We note that, in this case, the problem of  applying our techniques becomes tractable.  As shown in section 4, when the normal variational equation has an abelian Lie algebra, the abelianity of $\mathfrak{g}$ depends on whether integrals belong to $k$, the latter depending on whether some residues are null or not: this problem should be decidable. So, we believe that for families of parametrized Hamiltonian systems, once the normal variational equation has been fully reduced (which is not the contribution of this paper), the remaining part of the reduction should be tractable even in the presence of parameters.
\\
\section*{Appendix: Symplectic Gram-Schmidt Algorithm and Symplectic Linear Differential Systems}\label{GSS}
\def\thesection{A}
The material in this appendix is mostly well-known and included for the sake of the exposition's clarity. %

\subsection{A symplectic Gram-Schmidt method}\label{GS}

Let $(V,\omega)$ denote a symplectic vector space of dimension $2n$ with a basis 
${\bf u}:=\{u_{1}$,\ldots,$u_{2n}\}$. 
We briefly review a standard construction sometimes called \emph{symplectic Gram-Schmidt method}
(see e.g \cite{AM78}, chapter 3.1) for computing a symplectic (or Darboux) basis ${\bf e}:=\{e_{1}$,\ldots,$e_{2n}\}$  from ${\bf u}$,
i.e one on which the matrix of $\omega$ is $J$. 

Let $e_{1}:=u_{1}$. As $\omega$ is non-degenerate, one of the $u_{i}$, say $u_{2}$  satisfies 
$w(e_{1},u_{i})\neq 0$ hence we may set $e_{n+1}:=\frac{1}{\omega(e_{1},u_{2})}u_{2}$ 
so that $\omega(e_{1},e_{n+1})=1$.
Let $V_{1}:=\mathrm{span}_{C}\lbrace e_{1},e_{n+1}\rbrace$ and $V_{2}:=V_{1}^{\bot_{\omega}}$ be its symplectic orthogonal.
Then $V_{1}\bigcap V_{2}=\{0\}$ and a basis of $V_{2}$ is given by 
	$v_{i-2}:=u_{i}-\omega(u_{i},e_{n+1})e_{1}+\omega(u_{i},e_{1})e_{n+1}$ (for $i=3,\ldots,2n$)
so $V=V_{1}\oplus V_{2}$ and we may apply recursively the above to the basis $v_{1},\ldots,v_{n-2}$ of $V_{2}$. By construction, the 
result is a symplectic basis.	

\subsection{Symplectic linear differential systems}

Let  $[A]:\; Y'=AY$ with  $A\in \mathcal{M}_{2n} (k)$ be a linear differential  system such that $Gal([A])\subset Sp(2n,C)$. 
We may transform $A$ into a gauge-equivalent matrix $B\in \mathfrak{sp}(2n,C)$ as follows.\\
To $[A]$ is associated a differential module $({\cal M},\partial)$
where ${\cal M}=k^{2n}$ and  the action of a basis  and the action of $\partial$ over a basis $\epsilon=\{\epsilon_{1},\ldots,\epsilon_{2n}\}$ is given by $\partial(\epsilon)=-{}^{t}A\epsilon$.
This way, for $Y=\sum_{i} y_{i}\epsilon_{i}\in {\cal M}$,
$\partial{Y}=Y'-AY$. As $G\subset Sp(2n,C)$, there exists a vector 
$w=(\ldots,w_{i,j},\ldots)\in \bigwedge^{2}{ {\cal M}}$ (the coordinates $w_{i,j}\in k$ are expressed in the basis $(\epsilon_{i}\wedge \epsilon_{j})$)
such that $\partial(w)=0$. The matrix $\Omega=(\omega_{i,j})$ is a skew-symmetric non degenerate matrix (with the convention $w_{i,i}=0$). This $\Omega$ is the matrix of a symplectic form $\omega$ on ${\cal M}$. Applying the above symplectic Gram-Schmidt  method,  
a symplectic basis ${\bf e}=\{e_{1},\ldots,e_{n}\}$ of $\cal M$ is obtained.
Let $B$ be the matrix of $\partial$ on this basis. As the matrix of $\omega$ on ${\bf e}$ is $J$, the matrix of the (musical) isomorphism $\omega^\flat$ between $\cal M$ and ${\cal M}^{\star}$ is still $J$ so that $JB+{}^{t}BJ=0$.
\\

The construction of Normal Variational Equations in section \ref{222} admits a simple explanation in this formalism. 
We assume that $A\in \mathfrak{sp}(2n\,,\, C)$ so that the matrix of the symplectic form on $\mathcal{M}$ is $J$ in this basis. If we know a particular solution $Y_1 ={}^t (y_1 ,\ldots , y_{2n})\in k^{2n}$ of $Y'=AY$, 
we obtain what we call the (algebraic) \emph{Normal Variational Equation} by applying the Gram-Schmidt method to the basis $Y_1 , e_2,\ldots ,e_{2n}$ of $\mathcal{M}$. Notice that the Normal Variational Equation defined in this algebraic way doesn't necessarily coincide with the definition given for instance in \cite{MoSiSi05a}.

In our work, the construction implies that $y_{n+1}$ is by construction a constant vector : by duality,  $J\cdot  {}^t (1,\,0\,,\stackrel{2n-1}{\ldots},\, 0)$ is a solution of $Y'=-^{t}AY$ and therefore is a first integral of $Y'=AY$ (\cite{We95a,Mo99a}).  

\subsection{Symplectic solution matrices of symplectic linear differential systems}\label{Symplectic solution matrices of symplectic linear differential systems}

We consider again $[A]\,\,:\,\, Y' =AY$ with now $A\in \mathfrak{sp}(2n\,,\, C)$ and recall how to construct a symplectic fundamental solution matrix (this result is proved e.g in \cite{PuSi03a,MiSi02a}).
\begin{lem}
Let $k$ be a differential field and let $[A]\,:\, Y'=AY$ with $A\in\mathfrak{sp}(2n,C)$ be a linear Hamiltonian system. Let $K\supset k$ be the Picard-Vessiot extension of $k$ associated to $[A]$. Then, $[A]$ admits a symplectic  fundamental matrix of solutions $U\in\mathrm{Sp}(2n,K)$.
\end{lem}
\begin{proof}
Assume that we have a fundamental solution matrix $U$. Then $JU$ is a fundamental matrix of the dual system $[A^{\star}]\,\,:\,\, Y'=-^{t}A Y$ so there exists a constant matrix $\Phi$ such that  $JU ={}^{t}U^{-1}\Phi$. Therefore, $\Phi = {}^{t}UJU$ which is antisymmetric so it is the matrix of a symplectic form over the solution space of $[A]$ spanned by the columns $U_1 ,\ldots , U_{2n}$ of $U$. Applying the symplectic Gram Schmidt method to the basis $U_1 ,\ldots , U_{2n}$ (for the symplectic form $\omega (U_i \,,\, U_j):={}^{t}U \Phi U_j$) yields a symplectic basis $V_1 ,\ldots ,  V_{2n}$ and the matrix $\mathbf{V}=(V_1 , \ldots , V_{2n})$ is a fundamental solution matrix which satisfies by construction, ${}^{t}VJV =J$. 
\end{proof}

\begin{rem} As noted by a referee, there is a simpler proof when $k=C(x)$ or $k$ is some field of germs of holomorphic functions. Assume that $x=0$ is an ordinary point and construct the local fundamental solution matrix $U\in \mathcal{M}_n(C[[x]])$ such that $U(0)=\mathrm{Id}$.
Then, as ${}^t AJ+JA=0$, a calculation shows that $(^t U J U)'=0$ and, as $U(0)=\mathrm{Id}$, ${}^t U J U = J$. 
\end{rem}

Now we focus on the case where $2n=4$ and study the structure of symplectic matrices belonging to $\mathrm{Sp}(4,k) := \{\,U\in {\cal M}_{4}(k)\,:\, {}^t U J U = J\,\}$. 
The Lie algebra $\mathfrak{sp}(4 ,k)$ is the set of $4\times 4$ matrices $A\in M_{4}(k)$ such that
$$
A=\left[\begin{array}{cc} M & S_2 \\ S_1 & -{}^t M \end{array}\right]\quad\text{with}\quad M, S_i \in M_{2}(k)\quad \text{and}\quad {}^t S_i = S_i.
$$

Consider the linear differential system $[A]\,:\, X'=AX$ with $A\in\mathfrak{sp}(4,k)$. Suppose that it admits at least one rational solution $X_1 \in k^4$. Pick $P\in\mathrm{GL}(2n , k)$ such that  $X=PY$  and $P e_1 = X_1$ where $e_1:={}^t (1,0,0,0)$. 
This implies that  $Y_1 = e_1$ is a solution of the gauge equivalent system $[P[A]]$. 
Hence, the first column of $P[A]$ is null.
If, in addition, $A\in\mathfrak{sp}(4,k)$ and  $P\in\mathrm{Sp}(4,k)$, then $P[A]\in\mathfrak{sp}(4,k)$ and its form is

\begin{equation}\label{PlA}
P[A] = \left[\begin{array}{cccc} 0 & a & d & e \\ 0 & n_{11} & e & n_{12}\\ 0 & 0 & 0 & 0\\ 0 & n_{21} & -a & n_{11}\end{array}\right].
\end{equation}


\begin{prop}\label{FMS}
There is a symplectic fundamental matrix of solutions, $U\in\mathrm{Sp}(4,K)$ of the partially reduced system $(\ref{PlA})$ which has the following structure
$$
U=(u_{i,j})= \left[\begin{array}{cccc} 1 & q_{11}\Omega_1 - q_{12}\Omega_2 & \Omega_3 & q_{12}\Omega_1 + q_{22} \Omega_2 \\
 0 & q_{11}  &  \Omega_2 & q_{12}\\ 0 & 0 & 1 & 0 \\ 0 & q_{12} & -\Omega_1 & q_{22}\end{array}\right]
$$
where $Q:=(q_{i,j})$ is a unimodular fundamental matrix of solutions of the system $Y'=NY$, with $N:=(n_{i,j})\in\mathfrak{sp}(2\,,\, k)$.
\end{prop}
\begin{proof}
We can set without any loss of generality that the first column of $U:=(u_{i,j})$, denoted  $U_1$, is $e_1$. Since $U\in\mathrm{Sp}(4,K)$ ($K$ being the Picard Vessiot extension), we obtain  $\delta_{3j} = \omega (U_1 , U_j) = \omega(e_1, U_j) = u_{3j}$. Therefore, the matrix $U$ can be written in the form
$$
U= \left[\begin{array}{cccc} 1 & x & \Omega_3 & y\\
 0 & q_{11}  &  \Omega_2 & q_{12}\\ 0 & 0 & 1 & 0 \\ 0 & q_{12} & -\Omega_1 & q_{22}\end{array}\right]
$$
and the expression of $x$ and $y$ in terms of the remaining coefficients follows from the relations $\omega(U_2 , U_3) = \omega(U_3 , U_4)=0$. Furthermore,  since $\omega(U_2 , U_4) = 1 = \mathrm{det}(Q)$, $Q$ is unimodular and we are done.
\end{proof}




\end{document}